\documentclass[12pt,reqno,a4paper]{amsart}
\usepackage[margin=.5in]{geometry}
\usepackage{amsmath,amsthm,pdfsync,verbatim,graphicx,epstopdf,enumerate}
\usepackage{placeins}
\usepackage{amsfonts}
\usepackage{amssymb,amsthm,amsmath}
\usepackage{amsmath}
 \usepackage{enumerate}
\usepackage{ifthen}
\usepackage{graphicx}
\usepackage{mathtools}
\mathtoolsset{showonlyrefs}
\usepackage{cite}
\usepackage{tikz} 
\usetikzlibrary{arrows.meta}
\usepackage{comment}
\usepackage{amsmath,amscd,amssymb}
\usepackage{latexsym}
\usepackage[colorlinks,citecolor=blue,pagebackref,hypertexnames=false]{hyperref}

\numberwithin{equation}{section}
\allowdisplaybreaks[2]
\theoremstyle{plain}
\newtheorem{theorem}{Theorem}[section]

\newtheorem{lemma}[theorem]{Lemma}

\newtheorem{proposition}[theorem]{Proposition}

\theoremstyle{definition}
\newtheorem{definition}[theorem]{Definition}

\theoremstyle{remark}
\newtheorem{remark}[theorem]{Remark}

\newtheorem{case[theorem]}{Case}

\def\norm#1.#2.{\lVert#1\rVert_{#2}}

\def\M{\mathcal{M}}

\def \Ad{\mathcal{A}_{k,l}}

\def \M{{\mathcal M}}

\title[Hartee equation for anharmonic oscillator ]{Hartee-type heat equation associated to fractional anharmonic oscillator on weighted modulation spaces}

\author{Aparajita Dasgupta}
\author{Uttam Kumar Dolai}

 \address{\endgraf Department of Mathematics
 Indian Institute of Technology, Delhi, Hauz Khas
New Delhi-110016
 India}
\email{adasgupta@maths.iitd.ac.in}
\address{\endgraf Department of Mathematics
Indian Institute of Technology, Delhi, Hauz Khas
New Delhi-110016
India}
 \email{maz248176@maths.iitd.ac.in}


\date{\today}

\begin{document}


\begin{abstract}
We study Hartree-type nonlinear heat equations associated with fractional generalised anharmonic oscillators \(A_{k,l}\) in weighted modulation spaces. We first derive Strichartz-type estimates for the associated heat semigroup and then apply them to establish global well-posedness for small initial data. For \(s>\frac{d}{q'}\), the result is obtained via trilinear estimates exploiting the algebra property of the weighted modulation space, \(M^{p,q}_s\). We further establish refined trilinear estimates that bypass the algebra structure, thereby extending the global well-posedness theory to the wider range \(s\ge 0\).
\end{abstract}

	\maketitle

	\allowdisplaybreaks

	\tableofcontents 
\section{Introduction}

Nonlinear evolution equations involving nonlocal interactions have attracted considerable attention over the last several decades due to their fundamental role in mathematical physics, quantum mechanics, plasma dynamics, and many-body systems. Among these, the nonlinear Schr\"odinger equation with cubic convolution-type nonlinearity associated with the classical Laplacian has been studied extensively over the years. Such equations were first introduced by Ginibre and Velo \cite{MR562582} in the 1980s, motivated by the earlier work of Chadam and Glassey \cite{MR413843}, and are commonly referred to as the Hartree equation. The classical Hartree equation takes the form
\begin{align}
\label{Harteefordelta}
\iota u_{t}+\Delta u= (K*|u|^{2})u, 
\quad 
u(x,t_{0})=u_{0}(x),
\end{align}
where $u(x,t)$ is a complex-valued function on $\mathbb{R}^{d}\times \mathbb{R}$, $\Delta$ denotes the classical Laplacian on $\mathbb{R}^{d}$, $u_{0}$ is the prescribed initial datum, $K$ is a suitable potential defined on $\mathbb{R}^{d}$, and $*$  denotes the convolution operation.

Equation \eqref{Harteefordelta} arises naturally in the mean-field approximation of many-body quantum systems and has been extensively studied in connection with local and global well-posedness, regularity theory, scattering phenomena, blow-up behaviour, and long-time asymptotics; see, for instance, \cite{MR3195768, MR2002047, MR562582, SHYAM2026} and the references therein. Beyond the classical Laplacian setting, Hartree-type nonlinear equations associated with more general operators have also attracted considerable attention; see, for example, \cite{MR4944042, MR3114821, MR4027018, MR3128411}. Related wave equations involving Hartree nonlinearities have likewise been investigated extensively; see \cite{MR4896856, MR4742454, MR5042184}.

Nonlinear Schr\"odinger equations with Hartree-type nonlinearities have also been studied in the framework of modulation spaces; see, for example, \cite{MR4134386, MR3695957, MR3424616}. However, despite the extensive literature on Schr\"odinger-type models, the corresponding nonlinear heat equations remain comparatively less understood, particularly in the context of modulation spaces.

To the best of our knowledge, the well-posedness theory for nonlinear heat equations with Hartree-type interactions has not yet been developed in weighted modulation spaces. In particular, no results appear to be available for Hartree-type heat equations associated with generalised anharmonic oscillators in this setting.

The objective of the present work is to initiate this study. More precisely, we investigate nonlinear heat equations associated with fractional powers of the generalised anharmonic oscillator
\[
\mathcal{A}_{k,l}
:=
A(D)+V(x),
\]
where $A(\xi)$ and $V(x)$ are strictly positive homogeneous polynomials exhibiting growth of order $|\xi|^{2l}$ and $|x|^{2k}$, respectively. In particular, the operator $\mathcal{A}_{k,l}$ belongs to the $(l,k)$-class in the sense of \cite{MR4489248}.

 We consider the Hartree-type heat equation

\begin{align}
\label{harteeintro}
u_t+\Ad^\beta u=(K*|u|^2)u,
\qquad 
u(x,0)=u_0,
\end{align}
where $(x,t)\in \mathbb{R}^d\times [0,\infty)$, $\beta>0$, and the interaction kernel is given by

\[
K(x)=|x|^{-\gamma},
\qquad 
0<\gamma<d.
\]

The class of operators $\mathcal{A}_{k,l}$ forms a natural extension of the classical harmonic oscillator and is usually referred to as the class of \emph{generalised anharmonic oscillators}. Such operators arise naturally in quantum mechanical models whenever harmonic approximations become inadequate. For instance, in molecular vibrational dynamics, purely quadratic confinement often fails to describe accurately the behaviour of the energy spectrum. Higher-order polynomial potentials therefore provide a more realistic modelling framework, leading to substantially richer spectral and dynamical structures.

From the analytical viewpoint, anharmonic oscillators have been extensively investigated in spectral theory and microlocal analysis. For polynomial potentials of even degree, sharp eigenvalue asymptotics and spectral characterisations were obtained by Helffer and Robert \cite{MR657970,MR683006,MR662451}; see also \cite{MR743094,MR2304165,MR2599384} for further developments. More recently, nonsmooth anharmonic oscillators have attracted attention due to the additional analytical difficulties that arise beyond the classical smooth framework \cite{MR3554704}. A substantial generalisation was achieved in \cite{MR4489248}, where operators of the form, $A(D)+V(x)$ were studied within the framework of Weyl--H\"ormander calculus. The generalised anharmonic oscillators considered in the present article naturally belong to this setting, and the associated symbolic calculus and H\"ormander metrics can be developed analogously. Moreover, several operators studied previously in the literature, including those in \cite{MR4299820,MR4944933, dasgupta2026phase}, arise as particular cases of our framework.

Parallel to these developments, modulation spaces have emerged as an important functional setting for the study of evolution equations with rough initial data. Introduced by Feichtinger, modulation spaces capture simultaneously the spatial and frequency localisation properties of functions and have become fundamental objects in time-frequency analysis. Since initial data in modulation spaces are often considerably rougher than those in classical Sobolev or Bessel potential spaces, they provide an effective framework for studying low-regularity phenomena.

Over the past few decades, the classical Cauchy problem for the Hartree equation has been extensively studied in Sobolev spaces; see, for example, \cite{MR3195768, MR2002047, MR562582} and the references therein. The study of Hartree-type equations in modulation spaces has developed rapidly in recent years. Early contributions established local \cite{MR2506839} and global \cite{MR2281189} well-posedness results for initial data in $\M^{p,1}$. One of the key advantages of $\M^{p,1}$ lies in its multiplicative algebra property, which allows efficient treatment of nonlinear terms.  Subsequent works extended these results to the broader class of modulation spaces $\M^{p,q}$; see, for example, \cite{MR3977119,MR4944042,MR4017418}.

Nevertheless, to the best of our knowledge, the nonlinear heat equation \eqref{harteeintro} has not been investigated in weighted modulation spaces $\M_s^{p,q}$.
As a first step toward the analysis of \eqref{harteeintro}, we establish boundedness and smoothing estimates for the heat semigroup generated by \(\Ad^{\beta}\) on weighted modulation spaces. Estimates of this type were first obtained in \cite{MR4313961} for the Hermite operator in modulation spaces. Subsequently, Ruzhansky \emph{et al.} \cite{MR4944933} extended these results to anharmonic oscillators of the form
\[
(-\Delta)^l+|x|^{2k}.
\]

More recently, analogous semigroup estimates for operators of the form
\[
(-\Delta)^l+V(x),
\]
with general polynomial potentials were established in weighted modulation spaces in \cite{dasgupta2026phase}. In the present work, we derive corresponding estimates for the substantially broader class of generalised anharmonic oscillators \(\Ad\). As a consequence, our results encompass and recover all previously known estimates as particular cases.
 
 Building upon the fixed-time dispersive estimates established for the heat semigroup generated by the generalised anharmonic oscillator, we now derive Strichartz-type estimates in weighted modulation spaces. Such estimates have been extensively investigated for Schr\"odinger and wave equations in various functional settings; see, for instance, \cite{MR2350410, MR1646048, MR1351643, MR2060533}. In the context of modulation spaces, Strichartz estimates for Schr\"odinger equations were obtained in \cite{MR2992993, MR3065071}. 

These estimates play a fundamental role in the analysis of nonlinear evolution equations and serve as indispensable tools in the study of local and global existence, well-posedness in low-regularity Sobolev-type spaces, and scattering phenomena; see, for example, \cite{MR2257393, MR1778778, MR2195116}. In the present setting, they provide the analytical framework required to investigate the well-posedness theory of Hartree-type nonlinear heat equations associated with generalised anharmonic oscillators.

The derivation of Strichartz estimates for Schr\"odinger and wave equations is often based on the abstract framework developed by Keel and Tao \cite{MR1646048}. This approach applies since the corresponding solution operators form unitary groups on $L^{2}(\mathbb{R}^{d})$ and satisfy both the energy estimate and the dispersive decay estimate.

However, in the present setting, the solution operator associated with the heat equation generated by $\Ad^{\beta}$ is given by the heat semigroup, $e^{-t\Ad^{\beta}},$ which is not unitary on $L^{2}(\mathbb{R}^{d})$. Consequently, the abstract Strichartz theory of Keel and Tao cannot be applied directly. To overcome this difficulty, we combine the fixed-time dispersive estimates with Young's inequality to establish the required Strichartz-type estimates for the anharmonic heat semigroup.

Combining the boundedness properties of the anharmonic heat semigroup with the Strichartz-type estimates, we prove the existence and uniqueness of global solutions to \eqref{harteeintro} in weighted modulation spaces \(\M^{p,q}_{s}\). Our first approach relies on the multiplicative algebra property of \(\M^{p,q}_{s}\), which holds under the condition
$s>\frac{d}{q'}.$

This property enables us to derive the trilinear estimates required to control the Hartree-type convolution nonlinearity and to implement a contraction mapping argument for the global well-posedness theory. However, this approach imposes restrictions on the admissible range of the parameter \(p\). This naturally motivates the study of the complementary regime
$0\leq s\leq \frac{d}{q'}.$

To treat this case, we establish a new trilinear estimate that is independent of the algebra structure of modulation spaces. This allows us to obtain a second global well-posedness result in the weighted modulation spaces \(\M^{p,q}_{s}\), valid for a substantially broader range of the parameters \(s\) and \(p\).

The paper is organised as follows. In Section~2, we review the necessary preliminaries on modulation spaces, symbolic calculus, and the Weyl--H\"ormander framework. Section~3 is devoted to the boundedness properties of the heat semigroup generated by fractional generalised anharmonic oscillators on weighted modulation spaces. In Section~4, we derive Strichartz-type estimates for the associated anharmonic heat semigroup. Finally, Section~5 addresses the global well-posedness theory for Hartree-type nonlinear heat equations associated with fractional anharmonic operators.

\section{Preliminaries}
This section briefly reviews some fundamental concepts and properties of modulation spaces and the Weyl--H\"ormander calculus that will be used throughout the paper. For the definition and basic properties of modulation spaces, we refer the reader to \cite{feichtinger1983modulation, MR1843717}. In particular, further details related to the material presented in Section~2.1 can be found in \cite{dasgupta2026phase}. A comprehensive treatment of the Weyl--H\"ormander calculus can be found in the monograph of Nicola and Rodino \cite{MR2668420}. For further developments of the Weyl--H\"ormander framework in connection with anharmonic oscillators, we refer to \cite{MR4299820, MR4489248} and the references therein.

\subsection{Modulation Spaces}

We begin by recalling the definition of the short-time Fourier transform (STFT).
Let $f\in\mathcal{S}'(\mathbb{R}^d)$ and let $g\in\mathcal{S}(\mathbb{R}^d)\setminus\{0\}$
be a fixed window function. The STFT of $f$ with respect to $g$ is defined by
\[
V_g f(x,\xi)
= \int_{\mathbb{R}^d} f(z)\,\overline{g(z-x)}\,e^{-2\pi i \xi\cdot z}\,\mathrm{d}z,
\qquad (x,\xi)\in\mathbb{R}^{2d}.
\]

Modulation spaces were introduced by Feichtinger in the 1980s \cite{feichtinger1983modulation}
and constitute a fundamental tool in time-frequency analysis.

A weight function $v$ on $\mathbb{R}^{2d}$ is called \emph{submultiplicative} if
\[
v(x+y)\le v(x)\,v(y), \qquad x,y\in\mathbb{R}^{2d}.
\]
A weight $w$ is said to be \emph{$v$-moderate} if
\[
w(x+y)\le v(x)\,w(y), \qquad x,y\in\mathbb{R}^{2d}.
\]

Of particular importance in this article are the polynomial weights
\[
v_s(x,\xi)
:= \bigl(1+|x|^2+|\xi|^2\bigr)^{\frac{s}{2}},
\qquad s\in\mathbb{R}.
\]
Weights that are $v_s$-moderate for some $s\in\mathbb{R}$ are referred to as
\emph{polynomially moderate}.

Let $0<p,q\le\infty$ and let $w$ be a polynomially moderate weight.
The modulation space $\mathcal{M}^{p,q}_w(\mathbb{R}^d)$ consists of all
$f\in\mathcal{S}'(\mathbb{R}^d)$ such that
\[
\|f\|_{\mathcal{M}^{p,q}_w}
:= \|V_g f\, w\|_{L^{p,q}}
= \left(
\int_{\mathbb{R}^d}
\left(
\int_{\mathbb{R}^d}
|V_g f(x,\xi)\,w(x,\xi)|^p\,\mathrm{d}x
\right)^{\frac{q}{p}}
\mathrm{d}\xi
\right)^{\frac{1}{q}}
<\infty,
\]
with the usual modifications when $p=\infty$ or $q=\infty$.
Here $L^{p,q}(\mathbb{R}^d\times\mathbb{R}^d)$ denotes the mixed Lebesgue space
with (quasi-)norm
\[
\|F\|_{L^{p,q}} := \bigl\|\|F(x,\xi)\|_{L_x^p}\bigr\|_{L_\xi^q}.
\]

When the weight is exactly $v_s$, we use the shorthand notation
\[
\mathcal{M}^{p,q}_s := \mathcal{M}^{p,q}_{v_s}.
\]

Modulation spaces are quasi-Banach spaces (and Banach spaces when $p,q\ge1$),
and their definition is independent of the choice of window function $g$,
in the sense that different window functions yield equivalent norms.
We refer to
\cite{MR4286055,MR4201879,MR2028532,MR1843717}
for proofs and further details.

We conclude this subsection by recalling several properties of modulation spaces
that will be used throughout the paper.

\begin{lemma}[\cite{MR4849356}]
\label{Inclusionrelation}
Let $1\le p,p_1,p_2,q,q_1,q_2\le\infty$ and $s,s_1,s_2\in\mathbb{R}$. Then:
\begin{enumerate}
\item If $p,q<\infty$, then $\mathcal{S}(\mathbb{R}^d)$ is dense in
$\mathcal{M}^{p,q}_s(\mathbb{R}^d)$ and
\[
\bigl(\mathcal{M}^{p,q}_s(\mathbb{R}^d)\bigr)'
= \mathcal{M}^{p',q'}_{-s}(\mathbb{R}^d),
\]
where $p'$ and $q'$ denote the conjugate exponents of $p$ and $q$.

\item If $p_1\le p_2$, $q_1\le q_2$, and $s_1\ge s_2$, then
\[
\mathcal{M}^{p_1,q_1}_{s_1}(\mathbb{R}^d)
\hookrightarrow
\mathcal{M}^{p_2,q_2}_{s_2}(\mathbb{R}^d).
\]

\item If $s>\frac{d}{q'}$, or if $q=1$ and $s\ge0$, then
$\mathcal{M}^{p,q}_s(\mathbb{R}^d)\subset C(\mathbb{R}^d)$ and
$\mathcal{M}^{p,q}_s(\mathbb{R}^d)$ is a multiplicative algebra.
More precisely, there exists a constant $C>0$ such that
\[
\|fg\|_{\mathcal{M}^{p,q}_s}
\le C\,
\|f\|_{\mathcal{M}^{p,q}_s}\,
\|g\|_{\mathcal{M}^{p,q}_s},
\qquad
f,g\in\mathcal{M}^{p,q}_s(\mathbb{R}^d).
\]
\end{enumerate}
\end{lemma}

\subsection{A Modulation Weight Adapted to $\Ad$}

Following the approach adopted in \cite{MR4944933,dasgupta2026phase}, we introduce a modulation weight naturally associated with the operator \(\Ad\). Let \(k,l \geq 1\) be integers and consider the generalised anharmonic oscillator
\[
\Ad=A(D)+V(x),
\]
where the polynomials   $A(D), V(x)$ have growth of order $|\xi|^{2l} \text{ and } |x|^{2k}$ respectively.
We define a phase-space weight by
\[
\widetilde v(x,\xi)
:= (q_{1}+V(x)+A(\xi))^{1/2},
\qquad (x,\xi)\in\mathbb R^d\times\mathbb R^d,
\]
where $q_{1}>0$ is fixed. Using a similar argument as in \cite{dasgupta2026phase} it can be shown that 
for all $x,y,\xi,\eta\in\mathbb R^d$, there exists a constant $C>1$ such that
\begin{align}
\widetilde v(x,\xi)\,\widetilde v(y,\eta)
\gtrsim \widetilde v(x+y,\xi+\eta).
\label{eq:submultiplicative-weight}
\end{align}
Setting $v:=C\,\widetilde v$ therefore yields a submultiplicative weight on phase-pace.
This property guarantees the stability of the associated modulation spaces under time-frequency
shifts and is essential for the boundedness of pseudodifferential operators arising from the
functional calculus of $\mathcal{A}_{k,l}$.

Throughout the paper we work with weighted modulation spaces $ \mathcal{M}^{p,q}_{s}$ defined
with respect to the weight $v_s=(q_{1}+V(x)+A(\xi))^{s/2}$, $s\in\mathbb R$. The choice of $v$ allows us to treat
fractional anharmonic semigroups within a unified phase-space framework, extending beyond
classical harmonic and Gaussian settings.

\subsection{Symbolic Calculus}

In this subsection, we recall several notions and symbol classes that will be used frequently in the sequel. We refer to \cite{MR2668420} for further details and background.

A positive continuous function $\Phi(x,\xi)$ on $\mathbb{R}^{2d}$ is called a
\emph{sublinear weight} if
\begin{equation}
1 \le \Phi(x,\xi) \lesssim 1+|x|+|\xi|,
\qquad x,\xi\in\mathbb{R}^d.
\end{equation}
It is called a \emph{temperate weight} if there exists $s>0$ such that
\begin{equation}\label{temperate}
\Phi(x+y,\xi+\eta)
\lesssim \Phi(x,\xi)\,(1+|y|+|\eta|)^s,
\qquad x,\xi,y,\eta\in\mathbb{R}^d.
\end{equation}
Products and real powers of temperate weights are again temperate, and one always has
\[
(1+|y|+|\eta|)^{-s}
\lesssim \Phi(y,\eta)
\lesssim (1+|y|+|\eta|)^s.
\]

The weights of interest in this work are of the form
\[
\Phi(x,\xi)
:= q_{1} + V(x) + A(\xi),
\]
where $q_{1}>0$ is chosen such that $\Phi(x,\xi)\ge 1$, and where $A(\xi), V(x)$ are  strictly positive  polynomials having growth of order $2l$ and $2k$ respectively. Proceeding as in \cite{dasgupta2026phase} we can prove that $\Phi$ is a temperate weight.
 Consequently, $\Phi^s$ is temperate for every $s\in\mathbb{R}$.

\begin{definition}[Symbol class]
\label{Symboldefinition}
Let $\Phi(x,\xi)$ and $\Psi(x,\xi)$ be sublinear, temperate weights, and let
$M(x,\xi)$ be a temperate weight. The symbol class $S(M;\Phi,\Psi)$ consists of
all functions $a\in C^\infty(\mathbb{R}^{2d})$ such that for every
$\alpha,\beta\in\mathbb{N}^d$,
\[
|\partial_\xi^\alpha \partial_x^\beta a(x,\xi)|
\lesssim
M(x,\xi)\,\Psi(x,\xi)^{-|\alpha|}\,\Phi(x,\xi)^{-|\beta|},
\qquad (x,\xi)\in\mathbb{R}^{2d}.
\]
\end{definition}

The family of seminorms
\[
\|a\|_{k,S(M;\Phi,\Psi)}
:=
\sup_{|\alpha|+|\beta|\le k}
\sup_{(x,\xi)\in\mathbb{R}^{2d}}
|\partial_\xi^\alpha \partial_x^\beta a(x,\xi)|
M(x,\xi)^{-1}\Psi(x,\xi)^{|\alpha|}\Phi(x,\xi)^{|\beta|},
\]
with $k\in\mathbb{N}$, endows $S(M;\Phi,\Psi)$ with a Fr\'echet space topology.

\begin{definition}
A symbol $a$ is called \emph{globally elliptic} in $S(M;\Phi,\Psi)$ if
$a\in S(M;\Phi,\Psi)$ and there exists $R>0$ such that
\begin{equation}\label{eqn1}
|a(x,\xi)| \gtrsim M(x,\xi),
\qquad |x|+|\xi|\ge R.
\end{equation}
\end{definition}

\begin{definition}
A symbol $a\in S(M;\Phi,\Psi)$ is called \emph{globally hypoelliptic} if there exists
a temperate weight $M_0(x,\xi)$ and $R>0$ such that
\begin{equation}\label{eqn2}
|a(x,\xi)| \gtrsim M_0(x,\xi),
\qquad |x|+|\xi|\ge R,
\end{equation}
and for all $\alpha,\beta\in\mathbb{N}^d$,
\begin{equation}\label{eqn3}
|\partial_\xi^\alpha \partial_x^\beta a(x,\xi)|
\lesssim
|a(x,\xi)|\,\Psi(x,\xi)^{-|\alpha|}\,\Phi(x,\xi)^{-|\beta|},
\qquad |x|+|\xi|\ge R.
\end{equation}
We denote this class by $\mathrm{Hypo}(M,M_0;\Phi,\Psi)$.
\end{definition}

\begin{remark}
\label{remark}
If \eqref{eqn2} holds with $M_0=M$, then $a$ is elliptic.
Conversely, every elliptic symbol in $S(M;\Phi,\Psi)$ is hypoelliptic, since
\eqref{eqn2} holds trivially and \eqref{eqn3} follows from the defining estimates
of the symbol class.
\end{remark}



Furthermore, we recall several basic definitions and notations from the Weyl--H\"ormander calculus in order to keep the presentation self-contained. For a detailed treatment of the theory and its applications, we refer the reader to \cite{MR2304165,MR2599384,MR1011988, dasgupta2026phase}.

The Weyl quantisation of a symbol
$a\in\mathcal{S}'(\mathbb{R}^d\times\mathbb{R}^d)$
is defined by
\[
a^{w}(x,D)u(x)
:= \frac{1}{(2\pi)^d}
\int_{\mathbb{R}^d}\int_{\mathbb{R}^d}
e^{i\langle x-y,\xi\rangle}
a\!\left(\frac{x+y}{2},\xi\right)
u(y)\,\mathrm{d}y\,\mathrm{d}\xi,
\qquad u\in\mathcal{S}(\mathbb{R}^d).
\]

More generally, for $t\in\mathbb{R}$, the $t$-quantisation of $a$ is given by
\[
a_t(x,D)u(x)
:= \frac{1}{(2\pi)^d}
\int_{\mathbb{R}^d}\int_{\mathbb{R}^d}
e^{i\langle x-y,\xi\rangle}
a(tx+(1-t)y,\xi)
u(y)\,\mathrm{d}y\,\mathrm{d}\xi.
\]
In particular, the Weyl quantisation corresponds to $t=\frac12$, that is,
$a_{\frac12}(x,D)=a^{w}(x,D)$, while the choice $t=1$ yields the
Kohn-Nirenberg quantisation,
\[
a(x,D)u(x)
:= a_1(x,D)u(x)
= \frac{1}{(2\pi)^d}
\int_{\mathbb{R}^d}
e^{i\langle x,\xi\rangle}
a(x,\xi)\,\widehat{u}(\xi)\,\mathrm{d}\xi.
\]

We now recall the notion of a H\"ormander metric, which will play a crucial
role in the analysis of the generalised anharmonic oscillator and its
fractional powers.

\begin{definition}[H\"ormander metric]\label{HM}
Let $X\in\mathbb{R}^{2d}$ and let $g_X(\cdot)$ be a positive definite quadratic
form on $\mathbb{R}^{2d}$. The family $g=(g_X)_{X\in\mathbb{R}^{2d}}$ is called a
\emph{H\"ormander metric} if the following conditions hold:
\begin{enumerate}
\item[\rm (I)] \textbf{Continuity (slowness).}
There exists $C>0$ such that
\[
g_X(X-Y)\le C^{-1}
\quad\Longrightarrow\quad
\left(\frac{g_X(T)}{g_Y(T)}\right)^{\pm1}\le C,
\]
for all $T\in\mathbb{R}^{2d}\setminus\{0\}$.

\item[\rm (II)] \textbf{Uncertainty principle.}
Let $\sigma(Y,Z)=\langle z,\eta\rangle-\langle y,\zeta\rangle$
denotes the symplectic form and define
\[
g_X^{\sigma}(T)
:= \sup_{W\neq0}\frac{\sigma(T,W)^2}{g_X(W)}.
\]
We require that
\[
\lambda_g(X)
:= \inf_{T\neq0}\left(\frac{g_X^{\sigma}(T)}{g_X(T)}\right)^{1/2}
\ge 1,
\qquad X\in\mathbb{R}^{2d}.
\]

\item[\rm (III)] \textbf{Temperateness.}
There exist $\overline{C}>0$ and $J\in\mathbb{N}$ such that
\[
\left(\frac{g_X(T)}{g_Y(T)}\right)^{\pm1}
\le \overline{C}\bigl(1+g_Y^{\sigma}(X-Y)\bigr)^J,
\]
for all $X,Y,T\in\mathbb{R}^{2d}$.
\end{enumerate}
\end{definition}

Associated with a H\"ormander metric $g$ is the \emph{Planck function}
\[
h_g(X)^2
:= \sup_{T\neq0}\frac{g_X(T)}{g_X^{\sigma}(T)},
\]
which satisfies $h_g(X)=\lambda_g(X)^{-1}$.

\begin{definition}[$g$-weight]\label{GW}
Let $M:\mathbb{R}^{2d}\to(0,\infty)$.
\begin{itemize}
\item $M$ is called \emph{$g$-continuous} if there exists $\tilde{C}>0$ such that
\[
g_X(X-Y)\le \tilde{C}^{-1}
\quad\Longrightarrow\quad
\left(\frac{M(X)}{M(Y)}\right)^{\pm1}\le \tilde{C}.
\]

\item $M$ is called \emph{$g$-temperate} if there exists $\tilde{C}>0$ and
$N\in\mathbb{N}$ such that
\[
\left(\frac{M(X)}{M(Y)}\right)^{\pm1}
\le \tilde{C}\bigl(1+g_Y^{\sigma}(X-Y)\bigr)^N.
\]
\end{itemize}
We say that $M$ is a \emph{$g$-weight} if it is both $g$-continuous and
$g$-temperate.
\end{definition}

\begin{definition}
Let $g$ be a H\"ormander metric and let $M$ be a $g$-weight.
The symbol class $S(M,g)$ consists of all functions
$\sigma\in C^\infty(\mathbb{R}^{2d})$ such that for every $k\in\mathbb{N}$
there exists $C_k>0$ with
\begin{equation}\label{inwhk}
|\sigma^{(k)}(X;T_1,\ldots,T_k)|
\le C_k\,M(X)\prod_{j=1}^k g_X^{1/2}(T_j),
\end{equation}
for all $X,T_1,\ldots,T_k\in\mathbb{R}^{2d}$.
\end{definition}

The smallest constant $C_k$ in \eqref{inwhk} defines a seminorm
$\|\sigma\|_{k,S(M,g)}$, and the corresponding family
$\{\|\cdot\|_{k,S(M,g)}\}_{k\in\mathbb{N}}$ endows $S(M,g)$ with a Fr\'echet space
structure.

\subsection{Hörmander Metric Associated with $\mathcal A_{k,l}$}

We now recall the Hörmander metric intrinsically adapted to the operator $\mathcal A_{k,l}$.
Define
\[
g:=g^{k,l}
= \frac{dx^2}{(q_{1}+V(x)+ A(\xi))^{1/k}}
+ \frac{d\xi^2}{(q_{1}+V(x)+A(\xi))^{1/l}},
\]
where $q_{1}>0$ is chosen so that $q_{1}+V(x)+A(\xi)\ge1$ for all
$(x,\xi)\in\mathbb R^d\times\mathbb R^d$.
This metric encodes the quasi-homogeneous interaction between the spatial and frequency
variables.

It was shown in \cite{MR4299820,MR4489248} that $g$ is a Hörmander metric. Moreover, the
function
\[
M(x,\xi):=q_{1}+V(x)+A(\xi)
\]
is a $g$-weight in the sense of Hörmander's symbolic calculus. Consequently, the symbol class
\[
S(M,g)=S\bigl(q_{1}+V(x)+A(\xi),\,g\bigr)
\]
is well suited for the analysis of pseudodifferential operators associated with
$\mathcal A_{k,l}$ and its fractional powers.

The corresponding uncertainty parameter is given by
\[
\lambda_g
= (q_{1}+V(x)+A(\xi))^{\frac{k+l}{2kl}},
\]
and the associated Planck function $h_g$ satisfies
\[
h_g^{-1}=\lambda_g \simeq v^{\frac{k+l}{kl}},
\]
which establishes a precise quantitative link between the phase-space geometry induced by $g$
and the modulation weight $v$.


As discussed in \cite{MR4489248}, an equivalent description of the Hörmander symbol class
$S((q_{1}+V(x)+A(\xi))^{m/2},g)$ can be given in terms of anisotropic derivative estimates.
A symbol $a\in C^\infty(\mathbb R^d\times\mathbb R^d)$ is said to belong to the class
$\Sigma^m_{k,l}$ if
\begin{align}
\label{symbolclassdef}
|\partial_x^\beta \partial_\xi^\alpha a(x,\xi)|
\le C_{\alpha,\beta}
\bigl(q_{1}+V(x)+A(\xi)\bigr)^{\frac{m}{2}-\frac{|\beta|}{2k}-\frac{|\alpha|}{2l}},
\end{align}
for all multi-indices $\alpha,\beta$ and all $(x,\xi)\in\mathbb R^d\times\mathbb R^d$.

By Proposition~4.2 in \cite{MR4489248}, these classes coincide, namely
\[
\Sigma^m_{k,l}
= S\bigl((q_{1}+V(x)+A(\xi))^{m/2},\,g\bigr).
\]
Equipped with the natural family of seminorms, $\Sigma^m_{k,l}$ is a Fréchet space.

\section{Boundedness of Heat Semigroup  of Fractional Anharmonic Oscillator in Weighted Modulation Spaces}
This section is devoted to obtain a fixed time-estimate on weighted modulation spaces for the heat semigroup generated by anharmonic oscillator $\Ad=A(D)+V(x)$. Such estimate has already been established in \cite{dasgupta2026phase} for the operator of the form $(-\Delta)^{l}+V(x)$ which is a particular case of the operator $\Ad$ with $A(\xi)=|\xi|^{2l}$. Though most of the results are just  verbatim with a minor notational changes to the corresponding results in \cite{dasgupta2026phase}, for the sake of completeness we prefer to present those results here as well but skipping a few details.

As observed in \cite{MR4299820}, the anisotropic symbol classes discussed in Section~2 fit naturally into the general framework $S(M;\Phi,\Psi)$ of Nicola--Rodino \cite{MR2668420}.
Let
\[
k_0 := \max\{k,l\},
\]
the metric $g=g^{(k,l)}$ satisfies
\[
g
\le
\frac{\mathrm{d}x^{2}}{(q_{1}+V(x)+A(\xi))^{1/k_0}}
+
\frac{\mathrm{d}\xi^{2}}{(q_{1}+V(x)+A(\xi))^{1/k_0}}.
\]
Define
\[
\Phi(x,\xi)=\Psi(x,\xi)=(q_{1}+V(x)+A(\xi))^{\frac{1}{2k_0}},
\qquad
M(x,\xi)=q_{1}+V(x)+A(\xi).
\]
As discussed above, $\Phi$ and $\Psi$ are sublinear, temperate weights, and $M$ is temperate. Consequently, one can consider the symbol class $S(M;\Phi,\Psi)$, and we obtain the continuous embedding
\begin{align}
\label{Symbolinclusion}
S\bigl((q_{1}+V(x)+A(\xi))^{\frac{m}{2}},g\bigr)
\subset
S\bigl(M^{\frac{m}{2}};\Phi,\Psi\bigr).
\end{align}

Indeed, if $a\in S\bigl((q_{1}+V(x)+A(\xi))^{\frac{m}{2}},g\bigr)$, then by \eqref{symbolclassdef},
\begin{align*}
|\partial_x^{\beta}\partial_\xi^{\alpha} a(x,\xi)|
\le
C_{\alpha,\beta}(q_{1}+V(x)+A(\xi))^{\frac{m}{2}-\frac{|\beta|}{2k_0}-\frac{|\alpha|}{2k_0}}
=
C_{\alpha,\beta}\,M^{\frac{m}{2}}\Psi^{-|\alpha|}\Phi^{-|\beta|},
\end{align*}
which proves \eqref{Symbolinclusion}.

Using Lemma~5.4 of \cite{MR4489248}, it follows that,
\[
\mathcal{A}_{k,l}(x,\xi)\in\Sigma^{2}_{k,l},
\]
and therefore
\[
\mathcal{A}_{k,l}(x,\xi)\in S(M;\Phi,\Psi).
\]
Finally, there exist constants $C>0$ and $R>0$ such that
\[
A(\xi)+V(x)
\ge
C\,(q_{1}+V(x)+A(\xi)),
\qquad
|x|+|\xi|\ge R,
\]
which shows that $\Ad$ is elliptic.

Therefore, by Remark~\ref{remark}, we conclude that
\[
\mathcal{A}_{k,l}(x,\xi)\in \mathrm{Hypo}(M,M_{0};\Phi,\Psi),
\]
where $M, \Phi, \Psi$ are as above and 
\[
M_{0}(x,\xi)=q_{1}+V(x)+A(\xi).
\]
As a consequence of Theorem~4.2.9 in \cite{MR2668420}, the closure $\overline{\mathcal{A}_{k,l}}$ on $L^{2}(\mathbb{R}^{d})$ possesses a purely discrete spectrum consisting of eigenvalues tending to $+\infty$. Furthermore, since both $A(D)$ and the potential $V(x)$ are strictly positive operators, every eigenvalue of $\overline{\mathcal{A}_{k,l}}$ is strictly positive. In particular, $0$ does not belong to the spectrum of $\overline{\mathcal{A}_{k,l}}$
and hence $\mathcal{A}_{k,l}$ is invertible on $L^{2}(\mathbb{R}^{d})$.

Moreover, each eigenvalue of $\overline{\mathcal{A}_{k,l}}$ has finite multiplicity, and the associated eigenfunctions belong to the Schwartz class $\mathcal{S}(\mathbb{R}^{d})$. Consequently, the space $L^{2}(\mathbb{R}^{d})$ admits a complete orthonormal basis formed by eigenfunctions of $\overline{\mathcal{A}_{k,l}}$.

Let $\{\lambda_j\}_{j\geq 0}$ denote the sequence of eigenvalues of $\mathcal{A}_{k,l}$, arranged in nondecreasing order and counted according to multiplicity, with $\lambda_0>0$ representing the smallest eigenvalue. Due to the polynomial growth of $A(\xi)$ and $V(x)$, together with the quasi-homogeneous nature of the symbol $A(\xi)+V(x)$, the eigenvalue asymptotics established in Section~5 of \cite{MR4299820} for the harmonic oscillator $-\Delta+|x|^{2}$ extend naturally to the operator $\mathcal{A}_{k,l}$. More precisely, there exists a constant $C_{k,l}>0$, depending only on $k$ and $l$, such that
\begin{equation}
\label{eigenvalueestimate}
\lambda_j
\sim
C_{k,l}\, j^{\frac{2kl}{d(k+l)}},
\qquad \text{as } j\to\infty.
\end{equation}

Let $\{\Phi_{j,i}\}_{i=1}^{d_j}$ be an orthonormal basis of the eigenspace $H_j$ associated with the eigenvalue $\lambda_j$, where $d_j=\dim H_j$ denotes its multiplicity. Denoting by $P_j$ the orthogonal projection onto $H_j$, the spectral decomposition of $\mathcal{A}_{k,l}$ is given by
\begin{align*}
\mathcal{A}_{k,l}f
=
\sum_{j=0}^{\infty}\lambda_j P_jf,
\qquad
P_jf
=
\sum_{i=1}^{d_j}
\langle f,\Phi_{j,i}\rangle \Phi_{j,i},
\end{align*}
where $\langle\cdot,\cdot\rangle$ denotes the inner product on $L^{2}(\mathbb{R}^{d})$.

For any $\beta\in\mathbb{R}$, the fractional powers of $\mathcal{A}_{k,l}$ are defined through the spectral theorem as
\begin{align*}
\mathcal{A}_{k,l}^{\beta}f
=
\sum_{j=0}^{\infty}
\lambda_j^{\beta}P_jf.
\end{align*}

We next introduce a family of function spaces naturally adapted to the generalised anharmonic oscillator, referred to as the generalised anharmonic Sobolev spaces.

Let $\mathcal{A}_{k,l}$ denote the generalised anharmonic oscillator. For $s\in\mathbb{R}$ and $k,l\in\mathbb{N}$, we define the associated Sobolev space $Q^{s}_{k,l}$ by
\begin{align*}
Q^{s}_{k,l}
:=
\Bigl\{
u\in\mathcal{S}'(\mathbb{R}^{d})
:
\|(\mathcal{A}_{k,l})^{\frac{s}{2}}u\|_{L^{2}(\mathbb{R}^{d})}
<\infty
\Bigr\}.
\end{align*}
In the special case $A(\xi)=|\xi|^{2}$ and $V(x)=|x|^{2}$, the space $Q^{s}_{k,l}$ reduces to the classical Shubin--Sobolev space, also known as the Hermite--Sobolev space.

By means of the spectral decomposition of $\mathcal{A}_{k,l}$, the norm on $Q^{s}_{k,l}$ admits the equivalent representation
\begin{align*}
\|f\|^{2}_{Q^{s}_{k,l}}
:=
\|(\mathcal{A}_{k,l})^{\frac{s}{2}}f\|^{2}_{L^{2}(\mathbb{R}^{d})}
=
\sum_{j=0}^{\infty}
\lambda_{j}^{s}
\,
\|P_{j}f\|^{2}_{L^{2}(\mathbb{R}^{d})},
\end{align*}
where $\{\lambda_j\}_{j\geq 0}$ denotes the eigenvalues of $\mathcal{A}_{k,l}$ and $P_j$ the orthogonal projection onto the eigenspace corresponding to $\lambda_j$. Consequently, $Q^{s}_{k,l}$ may be characterised as the class of tempered distributions whose spectral coefficients satisfy a weighted $\ell^{2}$ summability condition determined by the spectral growth of $\mathcal{A}_{k,l}$.

Equipped with the sesquilinear form
\begin{align*}
\langle u,v\rangle_{Q^{s}_{k,l}}
:=
\Bigl\langle
(\mathcal{A}_{k,l})^{\frac{s}{2}}u,\,
(\mathcal{A}_{k,l})^{\frac{s}{2}}v
\Bigr\rangle_{L^{2}(\mathbb{R}^{d})},
\end{align*}
the space $Q^{s}_{k,l}$ becomes a Hilbert space. Moreover, its dual space can be naturally identified with $Q^{-s}_{k,l}$ through the $L^{2}$ pairing.

The proof of the following lemma follows verbatim from Lemma~4.4.19 in \cite{MR4201879}, upon replacing the standard weight by the modulation weight $v_s$ and invoking the embedding and duality properties of the spaces $Q^{s}_{k,l}$. We therefore omit the details.

\begin{lemma}
\label{anharmonicsobolevspace}
For every $s\in\mathbb{R}$,
\[
\mathcal{M}^{2,2}_{s}
=
Q^{s}_{k,l},
\]
with equivalence of norms.
\end{lemma}

As an immediate consequence of Lemma~\ref{anharmonicsobolevspace}, together with the inclusion relations of Lemma~\ref{Inclusionrelation} and Hölder's inequality, it follows that for every $0<p,q\leq \infty$ and sufficiently large $|s|$,
\begin{align*}
Q^{s}_{k,l}
\hookrightarrow
\mathcal{M}^{p,p}
\hookrightarrow
\mathcal{M}^{\infty,\infty}
\hookrightarrow
Q^{-s}_{k,l}.
\end{align*}

Throughout the sequel, $\mathcal{M}^{p,q}$ denotes the unweighted modulation space $\mathcal{M}^{p,q}_{s}$ corresponding to $s=0$. For nontrivial weights, the following refined embedding property holds.

\begin{lemma}
For sufficiently large parameters $s_{0}>s>0$, the following chain of continuous embeddings holds:
\begin{align}
\label{chain1}
Q^{s_{0}}_{k,l}
\hookrightarrow
\mathcal{M}^{p,q}_{s}
\hookrightarrow
\mathcal{M}^{\infty,\infty}_{s}
\hookrightarrow
Q^{-s_{0}}_{k,l}.
\end{align}
\end{lemma}

\begin{proof}
We first treat the case $p,q\geq 2$. By the inclusion relation in Lemma~\ref{Inclusionrelation}, for parameters $s_{0}>s>0$ there exists a continuous embedding
\[
\mathcal{M}^{2,2}_{s_{0}}
\hookrightarrow
\mathcal{M}^{p,q}_{s}.
\]
Invoking Lemma~\ref{anharmonicsobolevspace}, which yields the identification
\[
\mathcal{M}^{2,2}_{s_{0}}
=
Q^{s_{0}}_{k,l},
\]
we deduce that
\begin{align}
\label{chain}
Q^{s_{0}}_{k,l}
=
\mathcal{M}^{2,2}_{s_{0}}
\hookrightarrow
\mathcal{M}^{p,q}_{s}
\hookrightarrow
\mathcal{M}^{\infty,\infty}_{s}
\hookrightarrow
\mathcal{M}^{\infty,\infty}
\hookrightarrow
Q^{-s_{0}}_{k,l}.
\end{align}
This establishes \eqref{chain1} when $p,q\geq 2$.

We next consider the case $0<p,q\leq 2$. Without loss of generality, assume $p\leq q$. Since
\[
\mathcal{M}^{p,p}_{s}
\hookrightarrow
\mathcal{M}^{p,q}_{s},
\qquad q\geq p,
\]
it suffices to prove the embedding
\[
\mathcal{M}^{2,2}_{s_{0}}
\hookrightarrow
\mathcal{M}^{p,p}_{s}.
\]

Let $f\in \mathcal{M}^{2,2}_{s_{0}}$. By the definition of modulation spaces and Hölder's inequality, we have
\begin{align*}
\|V_{g}f\, v_{s}\|_{L^{p,p}}
&=
\|V_{g}f\, v_{s_{0}}\, v_{s-s_{0}}\|_{L^{p,p}} \\
&\leq
\|V_{g}f\, v_{s_{0}}\|_{L^{2,2}}
\,
\|v_{s-s_{0}}\|_{L^{r,r}},
\end{align*}
where the exponent $r$ is determined by
\[
\frac{1}{p}
=
\frac{1}{2}
+
\frac{1}{r}.
\]
Choosing $s_{0}$ sufficiently large guarantees that
\[
v_{s-s_{0}}
\in
L^{r,r},
\]
which yields the continuous embedding
\[
\mathcal{M}^{2,2}_{s_{0}}
\hookrightarrow
\mathcal{M}^{p,p}_{s}.
\]

Combining the two cases, we conclude that the chain of embeddings in \eqref{chain1} holds for all admissible values of $p$ and $q$, thereby completing the proof.
\end{proof}

We now state two key results useful for our work in this section. The first establishes boundedness of 
$t$-quantised operators associated with symbols in the anharmonic classes introduced above. This result is a direct consequence of Theorem~3.1 in \cite{MR3636061}, and its proof follows the same line of argument as Theorem~3.2 in \cite{MR4944933}. We therefore state it without proof.

\begin{theorem}
\label{quantisationtheorem}
Let $m\in\mathbb{R}$, $a\in\Sigma^{ -m }_{k,l}$, $0<p,q\leq \infty$, and $t\in\mathbb{R}$. Then the $t$-quantised operator $a_t(x,D)$ extends to a bounded operator
\[
a_t(x,D):\mathcal{M}^{p,q}\longrightarrow \mathcal{M}^{p,q}_{m},
\]
with operator norm depending only on finitely many seminorms of $a$ in $\Sigma^{ -m }_{k,l}$.
\end{theorem}
The following result shows that fractional powers of the generalised anharmonic oscillator admit a pseudodifferential representation with a positive elliptic symbol adapted to the anharmonic geometry. Since the argument is analogous to the proof of  \cite[Proposition~2.3]{MR4313961} and \cite[Theorem 3.3]{MR4944933}, with only minor modifications in our setting, we omit the details.

\begin{theorem}
\label{Calderontheorem}
Let $\beta>0,$ $A(\xi),V(x)$  are  two strictly positive homogeneous polynomials having growth of
order $2l$ and $2k$ respectively, and $k,l\ge 1$ be integers. Then the fractional anharmonic oscillator
\[
\mathcal{A}_{k,l}^{\beta}
:=
\bigl(A(D)+V(x)\bigr)^{\beta}
\]
is a pseudodifferential operator whose Weyl symbol $\Ad^{\beta}(x,\xi)$ is real-valued and belongs to the class $\Sigma^{2\beta}_{k,l}$. More precisely, the symbol admits the asymptotic representation
\begin{equation}
\label{assymtotic}
\mathcal{A}_{k,l}^{\beta}(x,\xi)
=
\bigl(A(\xi)+V(x)\bigr)^{\beta}
+
r(x,\xi),
\qquad
|\xi|^{l}+|x|^{k}\ge 1,
\end{equation}
where the remainder satisfies
\[
r \in \Sigma^{\,2\beta-\frac{k+l}{kl}}_{k,l}.
\]
\end{theorem}

As discussed above, for $\beta>0$ the operator $\mathcal{A}_{k,l}^{\beta}$ is a pseudodifferential operator with positive Weyl symbol belonging to the class $S(M^{\beta};\Phi,\Psi)$. Moreover, its spectrum consists of a discrete sequence of eigenvalues diverging to $+\infty$, and the associated eigenfunctions $\{\Phi_{j}\}_{j\geq 0}$ form an orthonormal basis of $L^{2}(\mathbb{R}^{d})$.

For $t\geq 0$, the semigroup generated by $\mathcal{A}_{k,l}^{\beta}$ is defined through the spectral theorem by
\[
e^{-t\mathcal{A}_{k,l}^{\beta}}f
=
\sum_{j=0}^{\infty}
e^{-t\lambda_{j}^{\beta}}
P_{j}f.
\]
By Theorem~4.5.1 in \cite{MR2668420}, the operator $e^{-t\mathcal{A}_{k,l}^{\beta}}$ is itself a pseudodifferential operator whose Weyl symbol satisfies a collection of uniform estimates. These estimates will be instrumental in establishing the boundedness and smoothing properties presented in the following theorem.

We are now ready to state the following result, which provides sharp boundedness and smoothing properties of the fractional anharmonic heat semigroup on weighted modulation spaces. The proof of the following theorem is essentially identical to that of Theorem 3.5 in \cite{dasgupta2026phase}. We therefore restrict our attention to the arguments that require modification and omit the rest of the proof.

\begin{theorem}
\label{BoundofH}
Let $\beta>0$, $0<p_{1},p_{2},q_{1},q_{2}\leq\infty$, $s_{1}>0$, and $s_{2}\in\mathbb{R}$. Define
\begin{align*}
\frac{1}{\widetilde{p}}
:=
\max\Bigl\{\frac{1}{p_{2}}-\frac{1}{p_{1}},\,0\Bigr\},
\qquad
\frac{1}{\widetilde{q}}
:=
\max\Bigl\{\frac{1}{q_{2}}-\frac{1}{q_{1}},\,0\Bigr\},
\end{align*}
and set
\begin{align}
\label{sigma}
\sigma
:=
\frac{d}{2\beta}
\left(
\frac{1}{k\,\widetilde{p}}
+
\frac{1}{l\,\widetilde{q}}
\right).
\end{align}
Then, for every $0<t\leq1$, the fractional anharmonic heat semigroup satisfies
\begin{equation}
\label{maininequality}
\bigl\|e^{-t\mathcal{A}_{k,l}^{\beta}}f\bigr\|_{\mathcal{M}^{p_{2},q_{2}}_{s_{2}}}
\le
C_{0}\,t^{-\sigma}\,
\|f\|_{\mathcal{M}^{p_{1},q_{1}}_{s_{1}}},
\end{equation}
where $C_{0}>0$ is a constant independent of both $t$ and $f$.

Moreover, in the case $s_{1}=s_{2}=s\ge0$, the estimate sharpens to
\begin{equation}
\label{secondconstant}
\bigl\|e^{-t\mathcal{A}_{k,l}^{\beta}}f\bigr\|_{\mathcal{M}^{p_{2},q_{2}}_{s}}
\le
\begin{cases}
C_{0}\,t^{-\sigma}\,\|f\|_{\mathcal{M}^{p_{1},q_{1}}_{s}}, & 0<t\le1,\\[0.3em]
C_{0}\,e^{-t\lambda_{0}^{\beta}}\,\|f\|_{\mathcal{M}^{p_{1},q_{1}}_{s}}, & t\ge1,
\end{cases}
\end{equation}
where $\lambda_{0}>0$ denotes the lowest eigenvalue of the anharmonic oscillator $\mathcal{A}_{k,l}$.
\end{theorem}
\begin{proof}
By the embedding properties of modulation spaces recalled earlier, it is sufficient to prove \eqref{maininequality} with $p_{2}$ replaced by $\min\{p_{1},p_{2}\}$ and $q_{2}$ replaced by $\min\{q_{1},q_{2}\}$. Therefore, without loss of generality, we may assume throughout the proof that
\[
p_{2}\leq p_{1}
\qquad \text{and} \qquad
q_{2}\leq q_{1}.
\]

Since $\mathcal{A}_{k,l}^{\beta}$, for $\beta>0$, is a pseudodifferential operator with positive elliptic Weyl symbol belonging to the class $S(M^{\beta};\Phi,\Psi)$, its $g$-ellipticity follows immediately from the asymptotic expansion \eqref{assymtotic}. Consequently, Theorem~4.5.1 of \cite{MR2668420} applies to the heat semigroup generated by $\mathcal{A}_{k,l}^{\beta}$.

In particular, the operator $e^{-t\mathcal{A}_{k,l}^{\beta}}$ is itself a pseudodifferential operator with Weyl symbol $b_t(x,\xi)$ depending on the parameter $t$. Moreover, for every $N\geq 0$, the family $\{t^{N}b_t\}_{t\in[0,T]}$ is uniformly bounded in the symbol class $\Sigma^{-2\beta N}_{k,l}$ for each fixed $T>0$. Consequently, an application of Theorem~\ref{quantisationtheorem} yields
\begin{align*}
\|e^{-t\mathcal{A}_{k,l}^{\beta}}f\|_{\mathcal{M}^{p_{1},q_{1}}}
\lesssim
\|f\|_{\mathcal{M}^{p_{1},q_{1}}}.
\end{align*}
Since $s_{1}>0$, the continuous embedding
\[
\mathcal{M}^{p_{1},q_{1}}_{s_{1}}
\hookrightarrow
\mathcal{M}^{p_{1},q_{1}}
\]
holds. Therefore,
\begin{align}
\label{modulationinclusion1}
\|V_{g}(e^{-t\mathcal{A}_{k,l}^{\beta}}f)\|_{L^{p_{1},q_{1}}}
&\lesssim
\|f\|_{\mathcal{M}^{p_{1},q_{1}}}
\nonumber\\
&\lesssim
\|f\|_{\mathcal{M}^{p_{1},q_{1}}_{s_{1}}}.
\end{align}
Next, since $t^{N}b_t\in \Sigma^{-2\beta N}_{k,l}$ for every $N\in\mathbb{N}$, an argument identical to the one above yields
\begin{align}
\label{modulationinclusion2}
\|t^{N}v_{2\beta N}
V_{g}(e^{-t\mathcal{A}_{k,l}^{\beta}}f)\|_{L^{p_{1},q_{1}}}
\lesssim
\|f\|_{\mathcal{M}^{p_{1},q_{1}}_{s_{1}}}.
\end{align}
Combining \eqref{modulationinclusion1} and \eqref{modulationinclusion2}, we obtain
\begin{align}
\label{eqn012}
\|(1+t^{N}v_{2\beta N})
V_{g}(e^{-t\mathcal{A}_{k,l}^{\beta}}f)\|_{L^{p_{1},q_{1}}}
\lesssim
\|f\|_{\mathcal{M}^{p_{1},q_{1}}_{s_{1}}},
\end{align}
uniformly for $t\in(0,1]$.

We now apply Hölder's inequality to estimate the weighted modulation norm. Writing
\begin{align}
\label{eqn1234}
\begin{aligned}
\|v_{s_{2}}
V_{g}(e^{-t\mathcal{A}_{k,l}^{\beta}}f)\|_{L^{p_{2},q_{2}}}
&=
\|v_{s_{2}}
(1+t^{N}v_{2\beta N})^{-1}
(1+t^{N}v_{2\beta N})
V_{g}(e^{-t\mathcal{A}_{k,l}^{\beta}}f)\|_{L^{p_{2},q_{2}}}
\\
&\lesssim
\|v_{s_{2}}
(1+t^{N}v_{2\beta N})^{-1}\|_{L^{\widetilde{p},\widetilde{q}}}
\\
&\qquad \times
\|(1+t^{N}v_{2\beta N})
V_{g}(e^{-t\mathcal{A}_{k,l}^{\beta}}f)\|_{L^{p_{1},q_{1}}},
\end{aligned}
\end{align}
where the exponents are determined by
\[
\frac{1}{p_{2}}
=
\frac{1}{p_{1}}
+
\frac{1}{\widetilde{p}},
\qquad
\frac{1}{q_{2}}
=
\frac{1}{q_{1}}
+
\frac{1}{\widetilde{q}}.
\]

Substituting \eqref{eqn012} into \eqref{eqn1234}, we deduce that
\begin{align}
\label{beforeenequality}
\|e^{-t\mathcal{A}_{k,l}^{\beta}}f\|_{\mathcal{M}^{p_{2},q_{2}}_{s_{2}}}
\lesssim
\|v_{s_{2}}
(1+t^{N}v_{2\beta N})^{-1}\|_{L^{\widetilde{p},\widetilde{q}}}
\,
\|f\|_{\mathcal{M}^{p_{1},q_{1}}_{s_{1}}}.
\end{align}
Thus, it remains to estimate the weighted mixed-norm quantity
\[
C(t)
:=
\|v_{s_{2}}
(1+t^{N}v_{2\beta N})^{-1}\|_{L^{\widetilde{p},\widetilde{q}}}.
\]
By the definition of the weight function $v$, we may write
\begin{align*}
C(t)
&=
\left\|
\frac{(q_{1}+V(x)+A(\xi))^{s_{2}/2}}
{1+t^{N}(q_{1}+V(x)+A(\xi))^{\beta N}}
\right\|_{L^{\widetilde{p},\widetilde{q}}}
\\
&\leq
\left\|
\frac{(q_{1}+V(x)+A(\xi))^{s_{2}/2}}
{1+t^{N}(V(x)+A(\xi))^{\beta N}}
\right\|_{L^{\widetilde{p},\widetilde{q}}}.
\end{align*}

Since the polynomials $A(\xi)$ and $V(x)$ satisfy
\[
A(\xi)\sim |\xi|^{2l},
\qquad
V(x)\sim |x|^{2k},
\]
for sufficiently large $|\xi|$ and $|x|$, there exists $R>0$ such that, whenever $|x|+|\xi|>R$,
\begin{align*}
&\left\|
\frac{(q_{1}+V(x)+A(\xi))^{s_{2}/2}}
{1+t^{N}(V(x)+A(\xi))^{\beta N}}
\right\|_{L^{\widetilde{p},\widetilde{q}}}^{\widetilde{q}}
\\
&\qquad\lesssim
\left\|
\frac{(1+|x|^{2k}+|\xi|^{2l})^{s_{2}/2}}
{1+t^{N}(|x|^{2k}+|\xi|^{2l})^{\beta N}}
\right\|_{L^{\widetilde{p},\widetilde{q}}}^{\widetilde{q}}.
\end{align*}

On the other hand, since the function
\[
(x,\xi)
\mapsto
\frac{(q_{1}+V(x)+A(\xi))^{s_{2}/2}}
{1+t^{N}(q_{1}+V(x)+A(\xi))^{\beta N}}
\]
is continuous on $\mathbb{R}^{2d}$, it is bounded on the compact region
\[
\{(x,\xi)\in\mathbb{R}^{2d}: |x|+|\xi|\leq R\}.
\]
Hence, the contribution from the compact part is finite and may be absorbed into the implicit constant.

Denoting
\[
I^{\widetilde{q}}
:=
\int_{\mathbb{R}^{d}}
\left(
\int_{\mathbb{R}^{d}}
\bigl(1+|\widetilde{x}|^{k}+|\widetilde{\xi}|^{l}\bigr)^{-(2\beta N-s_{2})\widetilde{p}}
\,\mathrm{d}\widetilde{x}
\right)^{\frac{\widetilde{q}}{\widetilde{p}}}
\mathrm{d}\widetilde{\xi},
\]
we note that \(I<\infty\) provided \(N\) is chosen sufficiently large. With this choice, the preceding estimate yields
\begin{align*}
C(t)^{\widetilde{q}}
\nonumber&\lesssim
I^{\widetilde{q}}\,
t^{-\frac{d}{2\beta }
\left(\frac{\widetilde{q}}{k\widetilde{p}}+\frac{1}{l}\right)}, 
\end{align*}
which is same as 
\begin{align}
\label{snot}
C(t)\le C'\,
t^{-\frac{d}{2\beta}
\left(\frac{1}{k\widetilde{p}}+\frac{1}{l\widetilde{q}}\right)},
\end{align}
for some constant $C'>0$ independent of $t\in(0,1]$. Consequently,
\[
C(t)
\lesssim
t^{-\sigma},
\qquad
0<t\le1,
\]
where
\[
\sigma
=
\frac{d}{2\beta}
\left(
\frac{1}{k\widetilde{p}}+\frac{1}{l\widetilde{q}}
\right).
\]
The next part of the proof follows along the same lines as that of \cite{dasgupta2026phase}. Therefore, we omit the details and conclude that
\begin{align}
\label{so}
\|e^{-t\mathcal{A}_{k,l}^{\beta}}f\|_{\mathcal{M}^{p_{2},q_{2}}_{s_{2}}}
\lesssim
t^{-\sigma}
\|f\|_{\mathcal{M}^{p_{1},q_{1}}_{s_{1}}}.
\end{align}


We now consider the special case $s_{1}=s_{2}=s\ge0$. Then the proof can be divided into two cases, namely $t\ge1$ and $0<t\le1$.

\medskip
\noindent
\textbf{Case 1: $0<t\leq 1$.}
In this case, for $s_{1}=s_{2}=s\ge 0$, the estimate \eqref{so} holds. The proof follows the same argument as in the case $s_{1}\neq s_{2}$.

\medskip
\noindent
\textbf{Case 2: \(t \geq 1\).} The proof in this case also follows along the same lines as the argument developed by the authors in \cite{dasgupta2026phase}. Therefore, we omit the details for the sake of brevity.

\end{proof}

\section{Strichartz-type Estimates for Fractional Anharmoinc Oscillator in Weighted Modulation Spaces }
In this section, we establish one of the main results of the paper, namely the Strichartz-type estimates associated with the heat semigroup generated by the fractional generalised anharmonic operator \(\Ad^{\beta}\).

To formulate these estimates, we introduce the space \(L^{r}_{t}([0,\infty), M^{p,q}_{s})\), consisting of all functions \(\psi:[0,\infty)\to M^{p,q}_{s}\) satisfying
\begin{align*}
\left\|\, \|\psi(t)\|_{M^{p,q}_{s}}\, \right\|_{L^{r}_{t}}<\infty.
\end{align*}
For convenience, we use the shorthand notation
\[
L^{r}_{t}\M^{p,q}:=L^{r}_{t}([0,\infty),M^{p,q}_{s}).
\]

We now establish the Strichartz-type estimates associated with the fractional anharmonic heat semigroup generated by \(\Ad^{\beta}\).
\begin{theorem}
\label{Strichartz}
Let \(1\leq p_{1},q_{1},p_{2},q_{2}\leq \infty\), \(1<r_{1},r_{2}<\infty\), and \(s\geq0\). Then the fractional heat semigroup generated by the generalised anharmonic oscillator satisfies the following Strichartz-type estimates:
\begin{enumerate}
\item[(i)] For
$\sigma<\frac{1}{r_{2}},$
we have
\begin{equation}
\label{1st strichartzconclusion}
\|e^{-t\Ad^{\beta}}f\|_{L^{r_{2}}_{t}\M^{p_{2},q_{2}}_{s}}
\lesssim
\|f\|_{\M^{p_{1},q_{1}}_{s}}.
\end{equation}
\item[(ii)] For
$\sigma<\frac{1}{r_{1}'}+\frac{1}{r_{2}},$
we have
\begin{equation}
\label{2nd strkzconclsn}
\left\|
\int_{0}^{t}
e^{-(t-\tau)\Ad^{\beta}}
F(\tau,\cdot)\,
d\tau
\right\|_{L^{r_{2}}_{t}\M^{p_{2},q_{2}}_{s}}
\lesssim
\|F\|_{L^{r_{1}}_{t}\M^{p_{1},q_{1}}_{s}},
\end{equation}
where \(\sigma\) is defined in \eqref{sigma}.
\end{enumerate}
\end{theorem}
\begin{proof}
Let \(f\in \M^{p_{1},q_{1}}_{s}\). By Theorem~\ref{BoundofH}, we have
\begin{align*}
\|e^{-t\Ad^{\beta}}f\|_{L^{r_{2}}_{t}\M^{p_{2},q_{2}}_{s}}
&=
\left\|
\|e^{-t\Ad^{\beta}}f\|_{\M^{p_{2},q_{2}}_{s}}
\right\|_{L^{r_{2}}_{t}}
\\
&\lesssim
\|C(t)\|_{L^{r_{2}}_{t}}
\|f\|_{\M^{p_{1},q_{1}}_{s}}
\\
&\lesssim
\left(
\int_{0}^{1}
t^{-\sigma r_{2}}\,dt
\right)^{\frac1{r_{2}}}
\|f\|_{\M^{p_{1},q_{1}}_{s}}
\\
&\qquad
+
\left(
\int_{1}^{\infty}
e^{-t\lambda_{0}^{\beta}r_{2}}\,dt
\right)^{\frac1{r_{2}}}
\|f\|_{\M^{p_{1},q_{1}}_{s}}.
\end{align*}
When \(\sigma r_{2}<1\), both integrals are finite. Therefore,
\[
\|e^{-t\Ad^{\beta}}f\|_{L^{r_{2}}_{t}\M^{p_{2},q_{2}}_{s}}
\lesssim
\|f\|_{\M^{p_{1},q_{1}}_{s}},
\]
which establishes the homogeneous estimate \eqref{1st strichartzconclusion}.

We next turn to the proof of the inhomogeneous estimate. Applying Theorem~\ref{BoundofH} and using Minkowski's integral inequality, we obtain
\begin{align*}
\left\|
\int_{0}^{t}
e^{-(t-\tau)\Ad^{\beta}}
F(\tau,\cdot)\,d\tau
\right\|_{L^{r_{2}}_{t}\M^{p_{2},q_{2}}_{s}}
&=
\left\|
\left\|
\int_{0}^{t}
e^{-(t-\tau)\Ad^{\beta}}
F(\tau,\cdot)\,d\tau
\right\|_{\M^{p_{2},q_{2}}_{s}}
\right\|_{L^{r_{2}}_{t}}
\\
&\le
\left\|
\int_{0}^{t}
\left\|
e^{-(t-\tau)\Ad^{\beta}}
F(\tau,\cdot)
\right\|_{\M^{p_{2},q_{2}}_{s}}
\,d\tau
\right\|_{L^{r_{2}}_{t}}
\\
&\lesssim
\left\|
\int_{0}^{t}
C(t-\tau)
\|F(\tau,\cdot)\|_{\M^{p_{1},q_{1}}_{s}}
\,d\tau
\right\|_{L^{r_{2}}_{t}},
\end{align*}

Applying Young's inequality, we obtain
\begin{equation}
\label{1stenq}
\left\|
\int_{0}^{t}
e^{-(t-\tau)\Ad^{\beta}}
F(\tau,\cdot)\,d\tau
\right\|_{L^{r_{2}}_{t}\M^{p_{2},q_{2}}_{s}}
\lesssim
\|C(t)\|_{L^{r}_{t}}
\|F\|_{L^{r_{1}}_{t}\M^{p_{1},q_{1}}_{s}},
\end{equation}
where the exponents satisfy
\[
1+\frac{1}{r_{2}}
=
\frac{1}{r}
+
\frac{1}{r_{1}},
\]
or equivalently,
\begin{equation}
\label{enqr}
\frac{1}{r}
=
\frac{1}{r_{1}'}
+
\frac{1}{r_{2}}.
\end{equation}

Next, by the definition of \(C(t)\),
\[
\|C(t)\|_{L^{r}_{t}}
\lesssim
\left(
\int_{0}^{1}
\tau^{-\sigma r}\,d\tau
\right)^{\frac1r}
+
\left(
\int_{1}^{\infty}
e^{-\tau\lambda_{0}^{\beta}r}\,d\tau
\right)^{\frac1r}.
\]
Hence \(\|C\|_{L^{r}_{t}}<\infty\) whenever
\[
\sigma r<1.
\]
Combining this with \eqref{enqr}, we obtain
\begin{equation}
\label{conditiononsgma}
\sigma<
\frac{1}{r_{1}'}
+
\frac{1}{r_{2}}.
\end{equation}

Therefore, under condition \eqref{conditiononsgma}, estimate \eqref{1stenq} yields the inhomogeneous Strichartz estimate \eqref{2nd strkzconclsn}. This completes the proof.
\end{proof}

\section{Global Well-posedness of Heat equation with Hartee-type Nonlinearity}
This section is devoted to the study of the well-posedness theory for nonlinear heat equations with Hartree-type nonlinearity associated with generalised anharmonic oscillators. More precisely, we consider the following Hartree-type evolution equation involving the fractional power of the generalised anharmonic operator \(\Ad^{\beta}\), where \(\beta>0\):
\begin{align}
\label{hartee1}
\begin{cases}
\partial_{t}u(t,x)+\Ad^{\beta}u(t,x)
=
\bigl(|x|^{-\gamma}*|u(t,x)|^{2}\bigr)u(t,x),\\
u(0,x)=u_{0}(x),
\end{cases}
\end{align}
where \(\gamma>0\) and \((t,x)\in [0,\infty)\times\mathbb{R}^{d}\).

The results established below will play a crucial role in developing the well-posedness results for \eqref{hartee1} in weighted modulation spaces.
\begin{lemma}
\label{gamaplemma}
Let \(0<\gamma<d\). Suppose that
\[
1<p_{1}<p_{2}<\infty,
\qquad
\frac{1}{p_{1}}+\frac{\gamma}{d}-1=\frac{1}{p_{2}},
\]
and let \(1\leq q\leq\infty\) and \(s>0\). Then, for every \(f\in\M^{p_{1},q}_{s}\),
\[
\||\cdot|^{-\gamma}*f\|_{\M^{p_{2},q}_{s}}
\lesssim
\|f\|_{\M^{p_{1},q}_{s}}.
\]
\end{lemma}

The proof follows directly from the Hardy--Littlewood--Sobolev inequality; we refer the reader to \cite{MR4134386} for details.

The following trilinear estimate constitutes a key ingredient in the contraction mapping argument used to establish the well-posedness result.
\begin{lemma}
\label{trlnrlma}
Let \(0<\gamma<d\), \(s>\frac{d}{q'}\), and let \(f,g,h\in \M^{p,q}_{s}\). Suppose that \(p_{2}\) satisfies
\[
1<p<p_{2}<\infty,
\qquad
\frac{1}{p}+\frac{\gamma}{d}-1=\frac{1}{p_{2}}.
\]
Then the following estimates hold:
\begin{equation}
\label{trlnr11}
\|(|\cdot|^{-\gamma}*(f\bar{g}))h\|_{\M^{p_{2},q}_{s}}
\lesssim
\|f\|_{\M^{p,q}_{s}}
\|g\|_{\M^{p,q}_{s}}
\|h\|_{\M^{p,q}_{s}}.
\end{equation}

Moreover,
\begin{align}
\label{trlnr2}
&\|(|\cdot|^{-\gamma}*|f|^{2})f-(|\cdot|^{-\gamma}*|g|^{2})g\|_{\M^{p_{2},q}_{s}}
\nonumber\\
&\qquad\lesssim
\Bigl(
\|f\|^{2}_{\M^{p,q}_{s}}
+
\|f\|_{\M^{p,q}_{s}}
\|g\|_{\M^{p,q}_{s}}
+
\|g\|^{2}_{\M^{p,q}_{s}}
\Bigr)
\|f-g\|_{\M^{p,q}_{s}}.
\end{align}
\end{lemma}
\begin{proof}
Since \(s>\frac{d}{q'}\), both \(\M^{p,q}_{s}\) and \(\M^{p_{2},q}_{s}\) posses the multiplicative algebra property. Moreover, the chosen exponents \(p\) and \(p_{2}\) satisfy the hypotheses of Lemma~\ref{gamaplemma}, and hence the lemma is applicable. Therefore,
\begin{align*}
\|(|\cdot|^{-\gamma}*(f\bar{g}))h\|_{\M^{p_{2},q}_{s}}
&\lesssim
\||\cdot|^{-\gamma}*(f\bar{g})\|_{\M^{p_{2},q}_{s}}
\|h\|_{\M^{p_{2},q}_{s}}
\\
&\lesssim
\|f\bar{g}\|_{\M^{p,q}_{s}}
\|h\|_{\M^{p_{2},q}_{s}}
\\
&\lesssim
\|f\|_{\M^{p,q}_{s}}
\|\bar{g}\|_{\M^{p,q}_{s}}
\|h\|_{\M^{p_{2},q}_{s}}.
\end{align*}
Since
\[
\M^{p,q}_{s}\hookrightarrow \M^{p_{2},q}_{s},
\]
and complex conjugation preserves the modulation space norm, the bound \eqref{trlnr11} follows as a direct consequence.
Next, to prove \eqref{trlnr2}, observe that
\begin{align*}
(|\cdot|^{-\gamma}*|f|^{2})f-(|\cdot|^{-\gamma}*|g|^{2})g
=
(|\cdot|^{-\gamma}*|f|^{2})(f-g)
+
(|\cdot|^{-\gamma}*(|f|^{2}-|g|^{2}))g.
\end{align*}
Hence,
\begin{align}
\nonumber
&\|(|\cdot|^{-\gamma}*|f|^{2})f-(|\cdot|^{-\gamma}*|g|^{2})g\|_{\mathcal{M}^{p_{2},q}_{s}} \\
\label{est0}
&\leq
\|(|\cdot|^{-\gamma}*|f|^{2})(f-g)\|_{\mathcal{M}^{p_{2},q}_{s}}
+
\|(|\cdot|^{-\gamma}*(|f|^{2}-|g|^{2}))g\|_{\mathcal{M}^{p_{2},q}_{s}}.
\end{align}

Applying \eqref{trlnr11}, we obtain
\begin{align}
\label{est1}
\|(|\cdot|^{-\gamma}*|f|^{2})(f-g)\|_{\mathcal{M}^{p_{2},q}_{s}}
\lesssim
\|f\|_{\mathcal{M}^{p,q}_{s}}^{2}
\|f-g\|_{\mathcal{M}^{p,q}_{s}},
\end{align}
and similarly,
\begin{align}
\label{est2}
\|(|\cdot|^{-\gamma}*(|f|^{2}-|g|^{2}))g\|_{\mathcal{M}^{p_{2},q}_{s}}
\lesssim
\||f|^{2}-|g|^{2}\|_{\mathcal{M}^{p,q}_{s}}
\|g\|_{\mathcal{M}^{p,q}_{s}}.
\end{align}

Now,
\begin{align}
\nonumber
\||f|^{2}-|g|^{2}\|_{\mathcal{M}^{p,q}_{s}}
&=
\|f(\Bar{f}-\Bar{g})+(f-g)\Bar{g}\|_{\mathcal{M}^{p,q}_{s}} \\
\nonumber
&\leq
\|f(\Bar{f}-\Bar{g})\|_{\mathcal{M}^{p,q}_{s}}
+
\|(f-g)\Bar{g}\|_{\mathcal{M}^{p,q}_{s}} \\
\nonumber
&\lesssim
\|f\|_{\mathcal{M}^{p,q}_{s}}
\|(\Bar{f}-\Bar{g})\|_{\mathcal{M}^{p,q}_{s}}
+
\|f-g\|_{\mathcal{M}^{p,q}_{s}}
\|\Bar{g}\|_{\mathcal{M}^{p,q}_{s}} \\
\label{est3}
&=
\|f\|_{\mathcal{M}^{p,q}_{s}}
\|f-g\|_{\mathcal{M}^{p,q}_{s}}
+
\|f-g\|_{\mathcal{M}^{p,q}_{s}}
\|g\|_{\mathcal{M}^{p,q}_{s}},
\end{align}
where the first inequality follows from the triangle inequality, the second from the algebra property of $\mathcal{M}^{p,q}_{s}$, and the final equality uses the invariance of the modulation norm under complex conjugation.

Finally, combining \eqref{est1}, \eqref{est2}, and \eqref{est3} in \eqref{est0}
yields the desired estimate \eqref{trlnr2}.
\end{proof}
Now we can prove the main result of this section.
\begin{theorem}
\label{welposedness1}
Let $0<\gamma<d$, $s>\frac{d}{q'}$, and
\[
1<p<\frac{d}{d-\gamma}.
\]
Assume that $r>0$ satisfies
\begin{align}
\label{r}
r>\frac{4\beta k}{2\beta k-d+\gamma}.
\end{align}
Then there exists $\varepsilon>0$ such that, whenever
\[
\|u_{0}\|_{\M^{p,q}_{s}}<\varepsilon,
\]
the nonlinear Hartree equation \eqref{hartee1} admits a unique global solution
\[
u\in L^{\infty}\bigl([0,\infty),\M^{p,q}_{s}\bigr)
   \cap L^{r}\bigl([0,\infty),\M^{p,q}_{s}\bigr).
\]
Moreover, the solution satisfies
\[
u\in C\bigl([0,\infty),\M^{p,q}_{s}\bigr)
   \cap L^{r}\bigl([0,\infty),\M^{p,q}_{s}\bigr).
\]
\end{theorem}
\begin{proof}
By the Duhamel principle, the equation \eqref{hartee1} can be written in the integral form
\begin{align}
\label{Duhamel}
\mathcal{J}(u)
:=e^{-t\Ad^{\beta}}u_{0}
+\int_{0}^{t}e^{-(t-\tau)\Ad^{\beta}}F(\tau,x)\,\mathrm{d}\tau,
\end{align}
where
\[
F(\tau,x)
=\bigl(|x|^{-\gamma}*|u(\tau,x)|^{2}\bigr)u(\tau,x),
\qquad
(\tau,x)\in [0,\infty)\times\mathbb{R}^{d}.
\]

We define the Banach space
\[
X
=
L^{\infty}\bigl([0,\infty),\M^{p,q}_{s}\bigr)
\cap
L^{r}\bigl([0,\infty),\M^{p,q}_{s}\bigr),
\]
equipped with the norm
\[
\|u\|_{X}
=
\|u\|_{L^{\infty}_{t}\M^{p,q}_{s}}
+
\|u\|_{L^{r}_{t}\M^{p,q}_{s}}.
\]

By Theorem \ref{Strichartz}, we obtain
\begin{align}
\label{nq1}
\|e^{-t\Ad^{\beta}}u_{0}\|_{L^{r}_{t}\M^{p,q}_{s}}
\lesssim
\|u_{0}\|_{\M^{p,q}_{s}}.
\end{align}
On the other hand, using Theorem \ref{BoundofH},
\begin{align*}
\|e^{-t\Ad^{\beta}}u_{0}\|_{L^{\infty}_{t}\M^{p,q}_{s}}
&=
\Bigl\|
\|e^{-t\Ad^{\beta}}u_{0}\|_{\M^{p,q}_{s}}
\Bigr\|_{L^{\infty}_{t}} \\
&\lesssim
\Bigl\|
C(t)\|u_{0}\|_{\M^{p,q}_{s}}
\Bigr\|_{L^{\infty}_{t}} \\
&=
\|C(t)\|_{L^{\infty}_{t}}
\|u_{0}\|_{\M^{p,q}_{s}},
\end{align*}
where
\[
C(t)=
\begin{cases}
C_{0}, & 0<t\leq 1,\\[0.3em]
C_{0}e^{-\lambda_{0}t}, & t\geq 1.
\end{cases}
\]
Since $\|C(t)\|_{L^{\infty}_{t}}<\infty$, it follows that
\begin{align}
\label{nq2}
\|e^{-t\Ad^{\beta}}u_{0}\|_{L^{\infty}_{t}\M^{p,q}_{s}}
\lesssim
\|u_{0}\|_{\M^{p,q}_{s}}.
\end{align}

Combining \eqref{nq1} and \eqref{nq2}, we arrive at
\begin{align}
\label{homo}
\|e^{-t\Ad^{\beta}}u_{0}\|_{X}
\leq
C_{1}\|u_{0}\|_{\M^{p,q}_{s}}.
\end{align}

To estimate the inhomogeneous term in \eqref{Duhamel}, we consider the following norm. By Theorem \ref{BoundofH},
\begin{align}
\label{inhm1}
\Bigl\|
\int_{0}^{t}e^{-(t-\tau)\Ad^{\beta}}F(\tau,\cdot)\,\mathrm{d}\tau
\Bigr\|_{L^{r}_{t}\M^{p,q}_{s}}
&\lesssim
\Bigl\|
\int_{0}^{t}
C(t-\tau)\,
\|F(\tau,\cdot)\|_{\M^{p_{2},q}_{s}}
\,\mathrm{d}\tau
\Bigr\|_{L^{r}_{t}}
\nonumber\\
&=
\Bigl\|
C(\tau)
*
\|F(\tau,\cdot)\|_{\M^{p_{2},q}_{s}}
\Bigr\|_{L^{r}_{t}}.
\end{align}

Here we choose $p_{2}$ satisfying
\[
1<p<p_{2}<\infty
\]
and
\begin{align}
\label{enq}
\frac{1}{p}+\frac{\gamma}{d}-1
=
\frac{1}{p_{2}}.
\end{align}
Since
\[
1<p<\frac{d}{d-\gamma},
\]
we have
\[
\frac{1}{p}+\frac{\gamma}{d}-1>0,
\]
which guarantees the existence of such a $p_{2}>1$.

In \eqref{inhm1}, the function $C(\tau)$ is given by
\[
C(\tau)=
\begin{cases}
\tau^{-\sigma}, & 0<\tau\leq 1,\\[0.3em]
e^{-\tau\lambda_{0}^{\beta}}, & \tau\geq 1,
\end{cases}
\]
where
\begin{align}
\label{sigma1}
\sigma
&=
\frac{d}{2\beta k}
\left(\frac{1}{p}-\frac{1}{p_{2}}\right) \nonumber\\
&=
\frac{d}{2\beta k}
\left(1-\frac{\gamma}{d}\right)
=
\frac{d-\gamma}{2\beta k}
>0.
\end{align}

Moreover, by Lemma \ref{trlnrlma},
\[
\|F(\tau,\cdot)\|_{\M^{p_{2},q}_{s}}
\lesssim
\|u(\tau,\cdot)\|_{\M^{p,q}_{s}}^{3}.
\]
Substituting this estimate into \eqref{inhm1}, we obtain
\[
\Bigl\|
\int_{0}^{t}
e^{-(t-\tau)\Ad^{\beta}}
F(\tau,\cdot)\,\mathrm{d}\tau
\Bigr\|_{L^{r}_{t}\M^{p,q}_{s}}
\lesssim
\Bigl\|
C(\tau)
*
\|u(\tau,\cdot)\|_{\M^{p,q}_{s}}^{3}
\Bigr\|_{L^{r}_{t}}.
\]

Next, choose $r_{2}>1$ such that
\begin{align}
\label{r2}
1+\frac{1}{r}
=
\frac{3}{r}+\frac{1}{r_{2}}.
\end{align}
Applying Young's inequality, we obtain
\begin{align*}
\left\|
\int_{0}^{t}
e^{-(t-\tau)\Ad^{\beta}}
F(\tau,\cdot)\,d\tau
\right\|_{L^{r}_{t}\M^{p,q}_{s}}
&\lesssim
\|C(t)\|_{L^{r_{2}}_{t}}
\left\|
\|u(t,\cdot)\|_{\M^{p,q}_{s}}^{3}
\right\|_{L^{r/3}_{t}}
\\
&\lesssim
\|C(t)\|_{L^{r_{2}}_{t}}
\|u\|_{L^{r}_{t}\M^{p,q}_{s}}^{3}.
\end{align*}
As in the proof of Theorem~\ref{Strichartz}, the integrability of  $C(t)$ is ensured by the condition
\[
\|C(t)\|_{L^{r_{2}}_{t}}<\infty,
\]
which holds whenever $\sigma r_{2}<1$. This is precisely guaranteed by our assumption on $r_{2}$. Indeed,
\begin{align*}
r>\frac{4\beta k}{2\beta k-d+\gamma}
&\iff
\frac{2}{r}
<
1-\frac{d}{2\beta k}
+\frac{\gamma}{2\beta k}
\\
&\iff
\frac{d}{2\beta k}
\left(1-\frac{\gamma}{d}\right)
<
1-\frac{2}{r}
=
\frac{1}{r_{2}}.
\end{align*}
Recalling the definition of $\sigma$, we have
\[
\sigma<\frac{1}{r_{2}},
\]
which immediately implies that $\sigma r_{2}<1$, and therefore
\begin{equation}
\label{inhm2}
\left\|
\int_{0}^{t}
e^{-(t-\tau)\Ad^{\beta}}
F(\tau,\cdot)\,d\tau
\right\|_{L^{r}_{t}\M^{p,q}_{s}}
\lesssim
\|u\|_{L^{r}_{t}\M^{p,q}_{s}}^{3}.
\end{equation}
We next estimate the inhomogeneous term in the space
\[
L^{\infty}\bigl([0,\infty),\M^{p,q}_{s}\bigr).
\]
Using Theorem \ref{BoundofH}, together with the same choice of exponents $p$ and $p_{2}$ as above, we obtain
\begin{align*}
\Bigl\|
\int_{0}^{t}
e^{-(t-\tau)\Ad^{\beta}}
F(\tau,\cdot)\,\mathrm{d}\tau
\Bigr\|_{\M^{p,q}_{s}}
&\leq
\int_{0}^{t}
\bigl\|
e^{-(t-\tau)\Ad^{\beta}}
F(\tau,\cdot)
\bigr\|_{\M^{p,q}_{s}}
\,\mathrm{d}\tau
\\
&\lesssim
\int_{0}^{t}
C(t-\tau)
\|F(\tau,\cdot)\|_{\M^{p_{2},q}_{s}}
\,\mathrm{d}\tau
\\
&\lesssim
\int_{0}^{t}
C(t-\tau)
\|u(\tau,\cdot)\|_{\M^{p,q}_{s}}^{3}
\,\mathrm{d}\tau
\\
&\lesssim
\|u\|_{L^{\infty}_{t}\M^{p,q}_{s}}^{3}
\int_{0}^{t}
C(t-\tau)\,\mathrm{d}\tau.
\end{align*}

It remains to estimate the integral involving $C(t)$. Observe that
\[
\int_{0}^{t}C(t-\tau)\,\mathrm{d}\tau
=
\begin{cases}
\displaystyle
\int_{0}^{t}(t-\tau)^{-\sigma}\,\mathrm{d}\tau,
& 0<t\leq 1,
\\[1em]
\displaystyle
\int_{0}^{t}
e^{-(t-\tau)\lambda_{0}^{\beta}}
\,\mathrm{d}\tau,
& t\geq 1.
\end{cases}
\]
Since $\sigma<\frac{1}{r_{2}}<1$, both integrals are uniformly bounded in $t$. 
Therefore,
\begin{align}
\label{inhm3}
\Bigl\|
\int_{0}^{t}
e^{-(t-\tau)\Ad^{\beta}}
F(\tau,\cdot)\,\mathrm{d}\tau
\Bigr\|_{L^{\infty}_{t}\M^{p,q}_{s}}
\lesssim
\|u\|_{L^{\infty}_{t}\M^{p,q}_{s}}^{3}.
\end{align}
Combining \eqref{inhm2} and \eqref{inhm3}, we deduce that
\begin{align}
\label{inhm}
\left\|
\int_{0}^{t}
e^{-(t-\tau)\Ad^{\beta}}
F(\tau,\cdot)\,d\tau
\right\|_{X}
&\lesssim
\|u\|_{L^{r}_{t}\M^{p,q}_{s}}^{3}
+
\|u\|_{L^{\infty}_{t}\M^{p,q}_{s}}^{3}
\nonumber\\
&\lesssim
\Bigl(
\|u\|_{L^{r}_{t}\M^{p,q}_{s}}
+
\|u\|_{L^{\infty}_{t}\M^{p,q}_{s}}
\Bigr)^{3}
\nonumber\\
&=
C_{2}\|u\|_{X}^{3},
\end{align}
for some constant \(C_{2}>0\).
Using \eqref{homo} together with \eqref{inhm}, we deduce that
\begin{align}
\label{neq1}
\|\mathcal{J}(u)\|_{X}
&\leq
C_{1}\|u_{0}\|_{\M^{p,q}_{s}}
+
C_{2}\|u\|_{X}^{3}
\nonumber\\
&\leq
C_{3}
\Bigl(
\|u_{0}\|_{\M^{p,q}_{s}}
+
\|u\|_{X}^{3}
\Bigr),
\end{align}
for some constant $C_{3}>0$.

For $\varepsilon>0$, define
\[
B_{\varepsilon}
=
\Bigl\{
u\in X
:\ 
\|u\|_{X}\leq \varepsilon
\Bigr\},
\]
which is the closed ball of radius $\varepsilon$ centered at the origin in $X$.

We next show that \(\mathcal{J}\) maps \(B_{\varepsilon}\) into itself for a suitable choice of \(\varepsilon>0\). Suppose that the initial datum satisfies
\[
\|u_{0}\|_{\M^{p,q}_{s}}
\leq
\frac{\varepsilon}{2C_{3}}.
\]
Then, for any \(u\in B_{\varepsilon}\), estimate \eqref{neq1} yields
\[
\|\mathcal{J}(u)\|_{X}
\leq
\frac{\varepsilon}{2}
+
C_{3}\varepsilon^{3}.
\]

Choosing \(\varepsilon>0\) sufficiently small so that
\[
C_{3}\varepsilon^{2}\leq \frac12,
\]
we obtain
\[
C_{3}\varepsilon^{3}
\leq
\frac{\varepsilon}{2}.
\]
Therefore,
\[
\|\mathcal{J}(u)\|_{X}
\leq
\frac{\varepsilon}{2}
+
\frac{\varepsilon}{2}
=
\varepsilon.
\]

It follows that \(\mathcal{J}(u)\in B_{\varepsilon}\) whenever \(u\in B_{\varepsilon}\). Hence \(\mathcal{J}\) maps \(B_{\varepsilon}\) into itself. 
{We now show that, for a sufficiently small choice of $\varepsilon>0$, the map $\mathcal{J}$ is a contraction on $B_{\varepsilon}$. The precise choice of $\varepsilon$ will be specified below.}

Let $u,v\in B_{\varepsilon}$. We shall prove that
\[
\|\mathcal{J}(u)-\mathcal{J}(v)\|_{X}
\leq
\frac{1}{2}\|u-v\|_{X}.
\]

Define
\[
F(t,x)
=
\bigl(|x|^{-\gamma}*|u(t,x)|^{2}\bigr)u(t,x),
\qquad
G(t,x)
=
\bigl(|x|^{-\gamma}*|v(t,x)|^{2}\bigr)v(t,x).
\]

Proceeding exactly as in the derivation of \eqref{inhm1}, and using the same choice of exponents $p$ and $p_{2}$, we obtain
\begin{align*}
\Bigl\|
\int_{0}^{t}
e^{-(t-\tau)\Ad^{\beta}}
\bigl(F(\tau,\cdot)-G(\tau,\cdot)\bigr)
\,\mathrm{d}\tau
\Bigr\|_{L^{r}_{t}\M^{p,q}_{s}}
\lesssim
\Bigl\|
C(\tau)
*
\|F(\tau,\cdot)-G(\tau,\cdot)\|_{\M^{p_{2},q}_{s}}
\Bigr\|_{L^{r}_{t}}.
\end{align*}

Let $r_{2}>1$ satisfy \eqref{r2}. By Young's inequality,
\begin{align*}
\Bigl\|
\int_{0}^{t}
e^{-(t-\tau)\Ad^{\beta}}
\bigl(F(\tau,\cdot)-G(\tau,\cdot)\bigr)
\,\mathrm{d}\tau
\Bigr\|_{L^{r}_{t}\M^{p,q}_{s}}
\lesssim
\|C(t)\|_{L^{r_{2}}_{t}}
\cdot
\Bigl\|
\|F(\tau,\cdot)-G(\tau,\cdot)\|_{\M^{p_{2},q}_{s}}
\Bigr\|_{L^{r/3}_{t}}.
\end{align*}

Since $\|C(t)\|_{L^{r_{2}}_{t}}<\infty$, it follows that
\begin{align}
\label{trlnr1}
\Bigl\|
\int_{0}^{t}
e^{-(t-\tau)\Ad^{\beta}}
\bigl(F(\tau,\cdot)-G(\tau,\cdot)\bigr)
\,\mathrm{d}\tau
\Bigr\|_{L^{r}_{t}\M^{p,q}_{s}}
\lesssim
\Bigl\|
\|F(\tau,\cdot)-G(\tau,\cdot)\|_{\M^{p_{2},q}_{s}}
\Bigr\|_{L^{r/3}_{t}}.
\end{align}

Next, by Lemma \ref{trlnrlma},
\begin{align*}
\|F(\tau,\cdot)-G(\tau,\cdot)\|_{\M^{p_{2},q}_{s}}
\lesssim
\Bigl(
\|u\|_{\M^{p,q}_{s}}^{2}
+
\|u\|_{\M^{p,q}_{s}}
\|v\|_{\M^{p,q}_{s}}
+
\|v\|_{\M^{p,q}_{s}}^{2}
\Bigr)
\|u-v\|_{\M^{p,q}_{s}}.
\end{align*}
Applying Hölder's inequality in time, we obtain
\begin{align*}
&
\Bigl\|
\|F(\tau,\cdot)-G(\tau,\cdot)\|_{\M^{p_{2},q}_{s}}
\Bigr\|_{L^{r/3}_{t}}
\\
&\lesssim
\Bigl\|
\Bigl(
\|u\|_{\M^{p,q}_{s}}^{2}
+
\|u\|_{\M^{p,q}_{s}}
\|v\|_{\M^{p,q}_{s}}
+
\|v\|_{\M^{p,q}_{s}}^{2}
\Bigr)
\|u-v\|_{\M^{p,q}_{s}}
\Bigr\|_{L^{r/3}_{t}}
\\
&\lesssim
\Bigl\|
\|u\|_{\M^{p,q}_{s}}^{2}
+
\|u\|_{\M^{p,q}_{s}}
\|v\|_{\M^{p,q}_{s}}
+
\|v\|_{\M^{p,q}_{s}}^{2}
\Bigr\|_{L^{r/2}_{t}}
\|u-v\|_{L^{r}_{t}\M^{p,q}_{s}}
\\
&\lesssim
\Bigl(
\|\|u\|_{\M^{p,q}_{s}}^{2}\|_{L^{r/2}_{t}}
+
\|\|u\|_{\M^{p,q}_{s}}
\|v\|_{\M^{p,q}_{s}}\|_{L^{r/2}_{t}}
+
\|\|v\|_{\M^{p,q}_{s}}^{2}\|_{L^{r/2}_{t}}
\Bigr)
\|u-v\|_{L^{r}_{t}\M^{p,q}_{s}}
\\
&\lesssim
\Bigl(
\|u\|_{L^{r}_{t}\M^{p,q}_{s}}^{2}
+
\|u\|_{L^{r}_{t}\M^{p,q}_{s}}
\|v\|_{L^{r}_{t}\M^{p,q}_{s}}
+
\|v\|_{L^{r}_{t}\M^{p,q}_{s}}^{2}
\Bigr)
\|u-v\|_{L^{r}_{t}\M^{p,q}_{s}}.
\end{align*}

Since $u,v\in B_{\varepsilon}$, estimate \eqref{trlnr1} yields
\begin{align*}
\Bigl\|
\int_{0}^{t}
e^{-(t-\tau)\Ad^{\beta}}
\bigl(F(\tau,\cdot)-G(\tau,\cdot)\bigr)
\,\mathrm{d}\tau
\Bigr\|_{L^{r}_{t}\M^{p,q}_{s}}
\leq
3C_{4}\varepsilon^{2}
\|u-v\|_{L^{r}_{t}\M^{p,q}_{s}},
\end{align*}
for some constant $C_{4}>0$.

Choosing $\varepsilon>0$ sufficiently small so that
\[
\varepsilon^{2}
<
\frac{1}{12C_{4}},
\]
we finally arrive at
\begin{align}
\label{123}
\Bigl\|
\int_{0}^{t}
e^{-(t-\tau)\Ad^{\beta}}
\bigl(F(\tau,\cdot)-G(\tau,\cdot)\bigr)
\,\mathrm{d}\tau
\Bigr\|_{L^{r}_{t}\M^{p,q}_{s}}
\leq
\frac{1}{4}
\|u-v\|_{L^{r}_{t}\M^{p,q}_{s}}.
\end{align}

Proceeding as in the proof of \eqref{inhm3}, we estimate
\begin{align}
\label{10}
\Bigl\|
\int_{0}^{t}
e^{-(t-\tau)\Ad^{\beta}}
\bigl(F(\tau,\cdot)-G(\tau,\cdot)\bigr)
\,\mathrm{d}\tau
\Bigr\|_{\M^{p,q}_{s}}
&\lesssim
\int_{0}^{t}
C(t-\tau)
\|F(\tau,\cdot)-G(\tau,\cdot)\|_{\M^{p_{2},q}_{s}}
\,\mathrm{d}\tau
\nonumber\\
&\lesssim
\|F(\tau,\cdot)-G(\tau,\cdot)\|_{L^{\infty}_{t}\M^{p_{2},q}_{s}}
\int_{0}^{t}
C(\tau)\,\mathrm{d}\tau
\nonumber\\
&\lesssim
\|F(\tau,\cdot)-G(\tau,\cdot)\|_{L^{\infty}_{t}\M^{p_{2},q}_{s}}.
\end{align}


Next, by Lemma \ref{trlnrlma},
\begin{align}
\label{12}
&
\|F(\tau,\cdot)-G(\tau,\cdot)\|_{L^{\infty}_{t}\M^{p_{2},q}_{s}}
\nonumber\\
&\lesssim
\Bigl\|
\Bigl(
\|u\|_{\M^{p,q}_{s}}^{2}
+
\|u\|_{\M^{p,q}_{s}}
\|v\|_{\M^{p,q}_{s}}
+
\|v\|_{\M^{p,q}_{s}}^{2}
\Bigr)
\|u-v\|_{\M^{p,q}_{s}}
\Bigr\|_{L^{\infty}_{t}}
\nonumber\\
&\lesssim
\Bigl(
\|u\|_{L^{\infty}_{t}\M^{p,q}_{s}}^{2}
+
\|u\|_{L^{\infty}_{t}\M^{p,q}_{s}}
\|v\|_{L^{\infty}_{t}\M^{p,q}_{s}}
+
\|v\|_{L^{\infty}_{t}\M^{p,q}_{s}}^{2}
\Bigr)
\|u-v\|_{L^{\infty}_{t}\M^{p,q}_{s}}.
\end{align}

Since $u,v\in B_{\varepsilon}$, combining \eqref{10} and \eqref{12} gives
\begin{align*}
\Bigl\|
\int_{0}^{t}
e^{-(t-\tau)\Ad^{\beta}}
\bigl(F(\tau,\cdot)-G(\tau,\cdot)\bigr)
\,\mathrm{d}\tau
\Bigr\|_{L^{\infty}_{t}\M^{p,q}_{s}}
\leq
3C_{5}\varepsilon^{2}
\|u-v\|_{L^{\infty}_{t}\M^{p,q}_{s}},
\end{align*}
for some constant $C_{5}>0$.

Choosing $\varepsilon>0$ sufficiently small so that
\[
\varepsilon^{2}
<
\frac{1}{12C_{5}},
\]
we arrive at
\begin{align}
\label{124}
\Bigl\|
\int_{0}^{t}
e^{-(t-\tau)\Ad^{\beta}}
\bigl(F(\tau,\cdot)-G(\tau,\cdot)\bigr)
\,\mathrm{d}\tau
\Bigr\|_{L^{\infty}_{t}\M^{p,q}_{s}}
\leq
\frac{1}{4}
\|u-v\|_{L^{\infty}_{t}\M^{p,q}_{s}}.
\end{align}

Combining \eqref{123} and \eqref{124}, we deduce that
\[
\Bigl\|
\int_{0}^{t}
e^{-(t-\tau)\Ad^{\beta}}
\bigl(F(\tau,\cdot)-G(\tau,\cdot)\bigr)
\,\mathrm{d}\tau
\Bigr\|_{X}
\leq
\frac{1}{2}
\|u-v\|_{X}.
\]
Hence, the mapping $\mathcal{J}$ is a contraction on $B_{\varepsilon}$.

Therefore, by Banach's contraction mapping theorem, the operator $\mathcal{J}$ admits a unique fixed point in $B_{\varepsilon}$. Consequently, the nonlinear Hartree equation \eqref{hartee1} possesses a unique global solution in
\[
L^{\infty}\bigl([0,\infty),\M^{p,q}_{s}\bigr)
\cap
L^{r}\bigl([0,\infty),\M^{p,q}_{s}\bigr).
\]

To prove that the unique solution belongs to
\[
C\bigl([0,\infty),\M^{p,q}_{s}\bigr)
\cap
L^{r}\bigl([0,\infty),\M^{p,q}_{s}\bigr),
\]
we first assume that the semigroup $e^{-t\Ad^{\beta}}$ is strongly
continuous on $\M^{p,q}_{s}$. Under this assumption, the fixed point
argument developed above can be repeated verbatim with
\[
C\bigl([0,\infty),\M^{p,q}_{s}\bigr)
\]
in place of
\[
L^{\infty}\bigl([0,\infty),\M^{p,q}_{s}\bigr),
\]
thereby yielding a solution which is continuous in time.

It therefore remains to justify the strong continuity of the semigroup
$e^{-t\Ad^{\beta}}$ on $\M^{p,q}_{s}$. In view of estimate \eqref{nq2},
it suffices to show that, for every $f$ in a dense subspace of
$\M^{p,q}_{s}$, the mapping
\[
t\longmapsto e^{-t\Ad^{\beta}}f
\]
is continuous with values in $\M^{p,q}_{s}$.

Since $\mathcal{S}(\mathbb{R}^{d})$ is dense in $\M^{p,q}_{s}$, let
$f\in\mathcal{S}(\mathbb{R}^{d})$. It is well known that the semigroup
$e^{-t\Ad^{\beta}}$ is strongly continuous on $L^{2}(\mathbb{R}^{d})$
and that $\Ad^{\gamma}$ commutes with $e^{-t\Ad^{\beta}}$ for every
$\gamma\in\mathbb{N}$. Hence, for every $\gamma\in\mathbb{N}$, the map
\[
t\longmapsto e^{-t\Ad^{\beta}}\Ad^{\gamma}f
=
\Ad^{\gamma}e^{-t\Ad^{\beta}}f
\]
is continuous with values in $L^{2}(\mathbb{R}^{d})$.

By a standard abstract result (see \cite[p.~194]{MR2668420}), the family
of seminorms
\[
p_{\gamma}(f)
:=
\|\Ad^{\gamma}f\|_{L^{2}(\mathbb{R}^{d})},
\qquad \gamma\in\mathbb{N},
\]
defines an equivalent system of seminorms on
$\mathcal{S}(\mathbb{R}^{d})$. Therefore, for any fixed
$t_{0}\in[0,\infty)$, the convergence
\[
e^{-t\Ad^{\beta}}f
\longrightarrow
e^{-t_{0}\Ad^{\beta}}f
\qquad
\text{in }\mathcal{S}(\mathbb{R}^{d})
\quad \text{as } t\to t_{0},
\]
is equivalent to
\[
\|\Ad^{\gamma}(e^{-t\Ad^{\beta}}f)
-
\Ad^{\gamma}(e^{-t_{0}\Ad^{\beta}}f)\|_{L^{2}(\mathbb{R}^{d})}
\longrightarrow 0
\qquad \text{as } t\to t_{0},
\]
for every $\gamma\in\mathbb{N}$. Since $\Ad^{\gamma}$ commutes with
$e^{-t\Ad^{\beta}}$, this is equivalent to
\[
\|e^{-t\Ad^{\beta}}\Ad^{\gamma}f
-
e^{-t_{0}\Ad^{\beta}}\Ad^{\gamma}f\|_{L^{2}(\mathbb{R}^{d})}
\longrightarrow 0
\qquad \text{as } t\to t_{0}.
\]

Thus, the mapping
\[
t\longmapsto e^{-t\Ad^{\beta}}f
\]
is continuous with values in $\mathcal{S}(\mathbb{R}^{d})$. Since the
embedding
\[
\mathcal{S}(\mathbb{R}^{d})
\hookrightarrow
\M^{p,q}_{s}
\]
is continuous, it follows that the same map is continuous as an
$\M^{p,q}_{s}$-valued function.

Consequently, for $p,q,r$ as in the hypothesis and
$s>\frac{d}{q'}$, the unique solution $u$ to \eqref{hartee1} satisfies
\[
u
\in
C\bigl([0,\infty),\M^{p,q}_{s}\bigr)
\cap
L^{r}\bigl([0,\infty),\M^{p,q}_{s}\bigr).
\]\end{proof}
In the previous theorem, the admissible ranges of the exponents
$p,q,r$ and the regularity index $s$ were constrained by several
technical assumptions. Our next goal is to relax these conditions and
establish global well-posedness for initial data in $\M^{p,q}_{s}$
for a broader range of parameters $p$ and $s$.
A key observation is that the condition
\[
s>\frac{d}{q'}
\]
ensures that the modulation space $\M^{p,q}_{s}$ is a multiplicative
algebra. This property played a fundamental role in establishing the trilinear estimate in Lemma~\ref{trlnrlma}, which in turn imposed the restrictions on the admissible ranges of the parameters \(p\) and \(q\) appearing in Theorem~\ref{welposedness1}.
To overcome these limitations, we shall no longer rely on the algebra
property of modulation spaces. Instead, we use the following multilinear
estimate, which allows greater flexibility in the choice of exponents.
We refer to
\cite{feichtinger1983modulation, MR2204673, MR2506839}
for its proof.

\begin{proposition}
\label{algebraproposition}
Let $m\in\mathbb{N}$ with $m\geq 1$, and assume that
\[
\sum_{i=1}^{m}\frac{1}{p_{i}}
=
\frac{1}{p_{0}},
\qquad
\sum_{i=1}^{m}\frac{1}{q_{i}}
=
m-1+\frac{1}{q_{0}},
\]
where
\[
0<p_{i}\leq\infty,
\qquad
1\leq q_{i}\leq\infty,
\qquad
1\leq i\leq m.
\]
Then there exists a constant $C>0$ such that
\[
\Bigl\|
\prod_{i=1}^{m}f_{i}
\Bigr\|_{\M^{p_{0},q_{0}}_{s}}
\leq
C
\prod_{i=1}^{m}
\|f_{i}\|_{\M^{p_{i},q_{i}}_{s}}.
\]
\end{proposition}

Using Proposition \ref{algebraproposition}, we establish the following
trilinear estimates.

\begin{lemma}
\label{Trlnerlma}
Let $0<\gamma<d$, $s\geq 0$, and let
\[
1\leq p,q,p_{0},q_{0}\leq \infty
\]
satisfy
\[
\frac{1}{p_{0}}
=
\frac{3}{p}
+\frac{\gamma}{d}
-1,
\qquad
\frac{1}{q_{0}}
=
3\Bigl(\frac{1}{q}-\frac{2}{3}\Bigr).
\]
Then, for all $f,g,h\in \M^{p,q}_{s}$, the Hartree-type trilinear operator obeys
\begin{align}
\label{Trlnr1}
\bigl\|
\bigl(|\cdot|^{-\gamma}*(f\overline{g})\bigr)h
\bigr\|_{\M^{p_{0},q_{0}}_{s}}
\lesssim
\|f\|_{\M^{p,q}_{s}}
\|g\|_{\M^{p,q}_{s}}
\|h\|_{\M^{p,q}_{s}}.
\end{align}
Moreover, the nonlinear mapping
\[
f \mapsto \bigl(|\cdot|^{-\gamma}*|f|^{2}\bigr)f
\]
is locally Lipschitz from $\M^{p,q}_{s}$ into $\M^{p_{0},q_{0}}_{s}$, in the sense that
\begin{align}
\label{Trlnr2}
&
\bigl\|
\bigl(|\cdot|^{-\gamma}*|f|^{2}\bigr)f
-
\bigl(|\cdot|^{-\gamma}*|g|^{2}\bigr)g
\bigr\|_{\M^{p_{0},q_{0}}_{s}}
\nonumber\\
&\lesssim
\Bigl(
\|f\|_{\M^{p,q}_{s}}^{2}
+
\|f\|_{\M^{p,q}_{s}}
\|g\|_{\M^{p,q}_{s}}
+
\|g\|_{\M^{p,q}_{s}}^{2}
\Bigr)
\|f-g\|_{\M^{p,q}_{s}}.
\end{align}
\end{lemma}
\begin{proof}
Let $p_{1}$ and $q_{1}$ be defined by
\[
\frac{1}{p_{1}}
=
\frac{2}{p}
+\frac{\gamma}{d}
-1,
\qquad
\frac{1}{q_{1}}
=
\frac{2}{q}
-1.
\]
Then
\begin{align}
\label{T101}
\frac{1}{p_{0}}
=
\frac{1}{p}
+
\frac{1}{p_{1}},
\qquad
1+\frac{1}{q_{0}}
=
\frac{1}{q}
+
\frac{1}{q_{1}}.
\end{align}

Hence, by Proposition \ref{algebraproposition},
\begin{align*}
\bigl\|
\bigl(|\cdot|^{-\gamma}*(f\overline{g})\bigr)h
\bigr\|_{\M^{p_{0},q_{0}}_{s}}
\lesssim
\bigl\|
|\cdot|^{-\gamma}*(f\overline{g})
\bigr\|_{\M^{p_{1},q_{1}}_{s}}
\|h\|_{\M^{p,q}_{s}}.
\end{align*}

Next, define $p_{3}$ by
\[
\frac{1}{p_{3}}
=
\frac{2}{p}.
\]
Then
\[
\frac{1}{p_{1}}
=
\frac{1}{p_{3}}
+
\frac{\gamma}{d}
-1.
\]
Therefore, by Lemma \ref{gamaplemma},
\begin{align*}
\bigl\|
|\cdot|^{-\gamma}*(f\overline{g})
\bigr\|_{\M^{p_{1},q_{1}}_{s}}
\lesssim
\|f\overline{g}\|_{\M^{p_{3},q_{1}}_{s}}.
\end{align*}

Observe that the exponents $p_{3}$ and $q_{1}$ satisfy
\[
\frac{1}{p_{3}}
=
\frac{1}{p}
+
\frac{1}{p},
\qquad
1+\frac{1}{q_{1}}
=
\frac{1}{q}
+
\frac{1}{q}.
\]
Applying Proposition \ref{algebraproposition} once again, we obtain
\begin{align*}
\|f\overline{g}\|_{\M^{p_{3},q_{1}}_{s}}
\lesssim
\|f\|_{\M^{p,q}_{s}}
\|g\|_{\M^{p,q}_{s}}.
\end{align*}
Here we used the fact that the modulation norm is invariant under
complex conjugation.

Combining the above estimates yields
\[
\bigl\|
\bigl(|\cdot|^{-\gamma}*(f\overline{g})\bigr)h
\bigr\|_{\M^{p_{0},q_{0}}_{s}}
\lesssim
\|f\|_{\M^{p,q}_{s}}
\|g\|_{\M^{p,q}_{s}}
\|h\|_{\M^{p,q}_{s}},
\]
which proves \eqref{Trlnr1}.
For the same choice of exponents $p_{1},q_{1}$, and $p_{3}$ as above, we write
\begin{align}
\label{1}
\bigl\|
\bigl(|\cdot|^{-\gamma}*|f|^{2}\bigr)f
-
\bigl(|\cdot|^{-\gamma}*|g|^{2}\bigr)g
\bigr\|_{\M^{p_{0},q_{0}}_{s}}
&\lesssim
\bigl\|
\bigl(|\cdot|^{-\gamma}*|f|^{2}\bigr)(f-g)
\bigr\|_{\M^{p_{0},q_{0}}_{s}}
\nonumber\\
&\quad+
\bigl\|
\bigl(|\cdot|^{-\gamma}*(|f|^{2}-|g|^{2})\bigr)g
\bigr\|_{\M^{p_{0},q_{0}}_{s}}
\nonumber\\
&\lesssim
\bigl\|
|\cdot|^{-\gamma}*|f|^{2}
\bigr\|_{\M^{p_{1},q_{1}}_{s}}
\|f-g\|_{\M^{p,q}_{s}}
\nonumber\\
&\quad+
\bigl\|
|\cdot|^{-\gamma}*(|f|^{2}-|g|^{2})
\bigr\|_{\M^{p_{1},q_{1}}_{s}}
\|g\|_{\M^{p,q}_{s}}
\nonumber\\
&\lesssim
\||f|^{2}\|_{\M^{p_{3},q_{1}}_{s}}
\|f-g\|_{\M^{p,q}_{s}}
\nonumber\\
&\quad+
\||f|^{2}-|g|^{2}\|_{\M^{p_{3},q_{1}}_{s}}
\|g\|_{\M^{p,q}_{s}}.
\end{align}
where in the last step we used Lemma \ref{gamaplemma}.

Since the exponents $p_{3}$ and $q_{1}$ satisfy the assumptions of
Proposition \ref{algebraproposition}, we obtain
\begin{equation}
\label{2}
\||f|^{2}\|_{\M^{p_{3},q_{1}}_{s}}
\lesssim
\|f\|_{\M^{p,q}_{s}}^{2}.
\end{equation}

Furthermore, by using a same line argument we can prove that 
\begin{align}
\label{3}
\||f|^{2}-|g|^{2}\|_{\M^{p_{3},q_{1}}_{s}}
&\lesssim
\|f-g\|_{\M^{p,q}_{s}}
\bigl(
\|f\|_{\M^{p,q}_{s}}
+
\|g\|_{\M^{p,q}_{s}}
\bigr).
\end{align}
In the last inequality we used the invariance of the modulation norm
under complex conjugation.
Finally, substituting \eqref{2} and \eqref{3} into \eqref{1}, we arrive at
\eqref{Trlnr2}.
\end{proof}
\begin{remark}
The existence of exponents $p_{0},q_{0}\geq 1$ satisfying the
assumptions of Lemma \ref{Trlnerlma} requires that
\[
0<
\frac{3}{p}
+\frac{\gamma}{d}
-1
<1,
\qquad
0<
3\Bigl(\frac{1}{q}-\frac{2}{3}\Bigr)
<1.
\]
Equivalently, the parameters $p$ and $q$ must satisfy
\begin{align}
\frac{3d}{2d-\gamma}
<
p
<
\frac{3d}{d-\gamma},
\qquad
1<q<\frac{3}{2}.
\end{align}

It is important to note that, for this range of exponents, the
trilinear estimates established above remain valid in
$\M^{p,q}_{s}$ for every $s\geq 0$, in contrast with
Lemma \ref{trlnrlma}, where the stronger condition
$s>\frac{d}{q'}$ was required.
\end{remark}

We are now in a position to establish the following global
well-posedness result under weaker assumptions on the modulation space
parameters.

\begin{theorem}
\label{welposedness2}

Let $s\geq 0$, $0<\gamma<d$, and let
\[
1<p,q,r<\infty
\]
satisfy
\[
\frac{3d}{2d-\gamma}
<
p
<
\frac{3d}{d-\gamma},
\qquad
1<q<\frac{3}{2},
\qquad
r>
\frac{2lq'\beta}{lq'\beta-d}.
\]
Then there exists $\varepsilon>0$ such that, whenever
\[
\|u_{0}\|_{\M^{p,q}_{s}}<\varepsilon,
\]
the nonlinear Hartree equation \eqref{hartee1} admits a unique global solution
\[
u
\in
L^{\infty}\bigl([0,\infty),\M^{p,q}_{s}\bigr)
\cap
L^{r}\bigl([0,\infty),\M^{p,q}_{s}\bigr).
\]
Moreover, the solution satisfies
\[
u
\in
C\bigl([0,\infty),\M^{p,q}_{s}\bigr)
\cap
L^{r}\bigl([0,\infty),\M^{p,q}_{s}\bigr).
\]
\end{theorem}

\begin{proof}
As in the proof of Theorem \ref{welposedness1}, we denote by
\[
X
=
L^{\infty}\bigl([0,\infty),\M^{p,q}_{s}\bigr)
\cap
L^{r}\bigl([0,\infty),\M^{p,q}_{s}\bigr),
\]
equipped with the norm
\[
\|u\|_{X}
=
\|u\|_{L^{\infty}_{t}\M^{p,q}_{s}}
+
\|u\|_{L^{r}_{t}\M^{p,q}_{s}},
\]
and by $\mathcal{J}$ the integral operator defined in
\eqref{Duhamel}.

Proceeding exactly as in the proof of Theorem
\ref{welposedness1}, we obtain
\begin{align}
\label{awk}
\|e^{-t\Ad^{\beta}}u_{0}\|_{X}
\leq
C_{1}'\|u_{0}\|_{\M^{p,q}_{s}},
\end{align}
for some constant $C_{1}'>0$.

Next, using Theorem \ref{BoundofH} together with the choice of
exponents $p_{0}$ and $q_{0}$ from Lemma \ref{Trlnerlma}, we estimate
\begin{align}
\label{qew}
&
\Bigl\|
\int_{0}^{t}
e^{-(t-\tau)\Ad^{\beta}}
F(\tau,\cdot)\,\mathrm{d}\tau
\Bigr\|_{L^{r}_{t}\M^{p,q}_{s}}
\nonumber\\
&\lesssim
\Bigl\|
\int_{0}^{t}
C(t-\tau)
\|F(\tau,\cdot)\|_{\M^{p_{0},q_{0}}_{s}}
\,\mathrm{d}\tau
\Bigr\|_{L^{r}_{t}}
\nonumber\\
&\lesssim
\Bigl\|
C(\tau)
*
\|u(\tau,\cdot)\|_{\M^{p,q}_{s}}^{3}
\Bigr\|_{L^{r}_{t}},
\end{align}
where
\[
C(\tau)
=
\begin{cases}
\tau^{-\sigma},
& 0<\tau\leq 1,
\\[0.3em]
e^{-\tau\lambda_{0}^{\beta}},
& \tau\geq 1.
\end{cases}
\]

The exponent $\sigma$ is given by
\begin{align}
\label{Sigma1}
\sigma
=
\frac{d}{2\beta}
\left(
\frac{1}{k\widetilde{p}}
+
\frac{1}{l\widetilde{q}}
\right).
\end{align}

Using \eqref{T101}, we compute
\[
\frac{1}{\widetilde{p}}
=
\max\Bigl\{
\frac{1}{p}-\frac{1}{p_{0}},
0
\Bigr\}
=
0,
\]
and
\[
\frac{1}{\widetilde{q}}
=
\max\Bigl\{
\frac{1}{q}-\frac{1}{q_{0}},
0
\Bigr\}
=
1-\frac{1}{q_{1}}
=
2-\frac{2}{q}
=
\frac{2}{q'}.
\]
Consequently, \eqref{Sigma1} reduces to
\begin{align}
\label{sigma2}
\sigma
=
\frac{d}{lq'\beta}.
\end{align}

Applying Young's inequality to \eqref{qew}, we obtain
\begin{align}
\label{qwee}
\Bigl\|
\int_{0}^{t}
e^{-(t-\tau)\Ad^{\beta}}
F(\tau,\cdot)\,\mathrm{d}\tau
\Bigr\|_{L^{r}_{t}\M^{p,q}_{s}}
\lesssim
\|C(t)\|_{L^{r_{2}}_{t}}
\Bigl\|
\|u\|_{\M^{p,q}_{s}}^{3}
\Bigr\|_{L^{r/3}_{t}},
\end{align}
where $r_{2}>1$ is chosen so that
\[
1+\frac{1}{r}
=
\frac{1}{r_{2}}
+
\frac{3}{r}.
\]

As in the proof of Theorem~\ref{welposedness1}, the quantity
\[
\|C(t)\|_{L^{r_{2}}_{t}}
\]
is finite whenever
\[
\sigma r_{2}<1.
\]
Since \(\frac{1}{r_{2}}=1-\frac{2}{r}\), this condition is equivalent to
\begin{equation}
\label{qwe1}
\sigma
<
1-\frac{2}{r}.
\end{equation}
It remains to verify that \eqref{qwe1} follows from the assumption on \(r\). Indeed, the condition
\[
r>
\frac{2lq'\beta}{lq'\beta-d}
\]
implies that
\[
\frac{2}{r}
<
1-\frac{d}{lq'\beta},
\]
and hence,
\[
\frac{d}{lq'\beta}
<
1-\frac{2}{r}.
\]
Invoking \eqref{sigma2}, we immediately obtain
\[
\sigma
<
1-\frac{2}{r},
\]
which is precisely condition \eqref{qwe1}.

Therefore, from \eqref{qwee}, we conclude that
\begin{align}
\label{asd}
\Bigl\|
\int_{0}^{t}
e^{-(t-\tau)\Ad^{\beta}}
F(\tau,\cdot)\,\mathrm{d}\tau
\Bigr\|_{L^{r}_{t}\M^{p,q}_{s}}
\lesssim
\|u\|_{L^{r}_{t}\M^{p,q}_{s}}^{3}.
\end{align}

Using the same argument as in the proof of Theorem
\ref{welposedness1}, together with the trilinear estimate
\eqref{Trlnr1}, we obtain
\begin{align}
\label{asd1}
\Bigl\|
\int_{0}^{t}
e^{-(t-\tau)\Ad^{\beta}}
F(\tau,\cdot)\,\mathrm{d}\tau
\Bigr\|_{L^{\infty}_{t}([0,\infty),\M^{p,q}_{s}}
\lesssim
\|u\|^{3}_{L^{\infty}_{t}([0,\infty),\M^{p,q}_{s}}.
\end{align}

Combining \eqref{asd} and \eqref{asd1}, we deduce that
\begin{align}
\label{awk1}
\Bigl\|
\int_{0}^{t}
e^{-(t-\tau)\Ad^{\beta}}
F(\tau,\cdot)\,\mathrm{d}\tau
\Bigr\|_{X}
\leq
C_{2}'\|u\|_{X}^{3},
\end{align}
for some constant $C_{2}'>0$.

Consequently, from \eqref{awk} and \eqref{awk1}, we arrive at
\begin{align}
\label{awk2}
\|\mathcal{J}(u)\|_{X}
\leq
C_{3}'
\Bigl(
\|u_{0}\|_{\M^{p,q}_{s}}
+
\|u\|_{X}^{3}
\Bigr),
\end{align}
for some constant $C_{3}'>0$.

Choose $\varepsilon>0$ such that
\[
\varepsilon^{2}
<
\frac{1}{2C_{3}'},
\]
and assume that the initial data satisfies
\[
\|u_{0}\|_{\M^{p,q}_{s}}
<
\frac{\varepsilon}{2C_{3}'}.
\]
Define
\[
B_{\varepsilon}
=
\Bigl\{
u\in X
:
\|u\|_{X}\leq \varepsilon
\Bigr\},
\]
which is a closed ball centered at the origin in $X$.
Then, by \eqref{awk2}, the operator $\mathcal{J}$ maps
$B_{\varepsilon}$ into itself.

{We next show that, for a sufficiently small choice of $\varepsilon>0$, the operator $\mathcal{J}$ is a contraction on $B_{\varepsilon}$.} Let
\[
F(t,x)
=
\bigl(|x|^{-\gamma}*|u(t,x)|^{2}\bigr)u(t,x),
\]
and
\[
G(t,x)
=
\bigl(|x|^{-\gamma}*|v(t,x)|^{2}\bigr)v(t,x).
\]
Using the estimate \eqref{Trlnr2}, together with Young's and
H\"older's inequalities, and arguing exactly as before, we obtain
\begin{align*}
&
\Bigl\|
\int_{0}^{t}
e^{-(t-\tau)\Ad^{\beta}}
\bigl(F(\tau,\cdot)-G(\tau,\cdot)\bigr)
\,\mathrm{d}\tau
\Bigr\|_{L^{r}_{t}\M^{p,q}_{s}}
\\
&\lesssim
\Bigl(
\|u\|_{L^{r}_{t}\M^{p,q}_{s}}^{2}
+
\|u\|_{L^{r}_{t}\M^{p,q}_{s}}
\|v\|_{L^{r}_{t}\M^{p,q}_{s}}
+
\|v\|_{L^{r}_{t}\M^{p,q}_{s}}^{2}
\Bigr)
\|u-v\|_{L^{r}_{t}\M^{p,q}_{s}},
\end{align*}
and similarly,
\begin{align*}
&
\Bigl\|
\int_{0}^{t}
e^{-(t-\tau)\Ad^{\beta}}
\bigl(F(\tau,\cdot)-G(\tau,\cdot)\bigr)
\,\mathrm{d}\tau
\Bigr\|_{L^{\infty}_{t}\M^{p,q}_{s}}
\\
&\lesssim
\Bigl(
\|u\|_{L^{\infty}_{t}\M^{p,q}_{s}}^{2}
+
\|u\|_{L^{\infty}_{t}\M^{p,q}_{s}}
\|v\|_{L^{\infty}_{t}\M^{p,q}_{s}}
+
\|v\|_{L^{\infty}_{t}\M^{p,q}_{s}}^{2}
\Bigr)
\|u-v\|_{L^{\infty}_{t}\M^{p,q}_{s}}.
\end{align*}

Since $u,v\in B_{\varepsilon}$, choosing $\varepsilon>0$
sufficiently small and combining the above two estimates, we obtain
\[
\|\mathcal{J}(u)-\mathcal{J}(v)\|_{X}
\leq
\frac{1}{2}
\|u-v\|_{X}.
\]
Hence $\mathcal{J}$ is a contraction on $B_{\varepsilon}$.

Therefore, by Banach's contraction mapping principle,
$\mathcal{J}$ admits a unique fixed point in $B_{\varepsilon}$,
which yields a unique global solution to \eqref{hartee1} in
\[
L^{\infty}([0,\infty),\M^{p,q}_{s})
\cap
L^{r}([0,\infty),\M^{p,q}_{s}),
\]
for all exponents $p,q,r,s$ satisfying the assumptions of the theorem.

Finally, arguing exactly as in the proof of Theorem
\ref{welposedness1}, we conclude that the solution moreover satisfies
\[
u
\in
C\bigl([0,\infty),\M^{p,q}_{s}\bigr)
\cap
L^{r}([0,\infty),\M^{p,q}_{s}).
\]
This completes the proof.
 \end{proof}

 \section*{Concluding Remarks}

In this paper, we investigated the fractional powers of the generalised anharmonic oscillator
\[
\mathcal{A}^{\beta}_{k,l}:= \left(A(D) + V(x)\right)^{\beta},
\]
where $\beta>0$ and $A(\xi), V(x)$ are strictly positive homogeneous polynomials with growth of order $|\xi|^{2l}$ and $|x|^{2k}$, respectively. We studied the associated semigroup $e^{-t\mathcal{A}_{k,l}^{\beta}}$ on weighted modulation spaces $\mathcal{M}^{p,q}_{s}$, $s\in\mathbb{R}$, for the full range $0<p,q\leq\infty$. In particular, we first observed estimates of the form
\[
\bigl\| e^{-t\mathcal{A}_{k,l}^{\beta}} f \bigr\|_{\mathcal{M}^{p_2,q_2}_{s_2}}
\le
C(t)\,
\|f\|_{\mathcal{M}^{p_1,q_1}_{s_1}},
\]
with explicit dependence of the constant $C(t)$ on time using similar techniques as in \cite{dasgupta2026phase, MR4944933, MR4313961}.

One of the important contributions of the present work is the establishment of Strichartz type estimates for the heat semigroup generated by $\mathcal{A}_{k,l}^{\beta}$ on weighted modulation spaces.

\begin{itemize}
\item For $\sigma<\frac{1}{r_{2}}$, we prove the homogeneous estimate
\begin{align*}
\|e^{-t\mathcal{A}_{k,l}^{\beta}}f\|_{L^{r_{2}}_{t}\mathcal{M}^{p_{2},q_{2}}_{s}}
\lesssim
\|f\|_{\mathcal{M}^{p_{1},q_{1}}_{s}}.
\end{align*}

\item We also establish the corresponding inhomogeneous estimate
\[
\sigma<\frac{1}{r_{1}'}+\frac{1}{r_{2}},
\]
namely,
\begin{align*}
\left\|
\int_{0}^{t}
e^{-(t-\tau)\mathcal{A}_{k,l}^{\beta}}
F(\tau,\cdot)\,d\tau
\right\|_{L^{r_{2}}_{t}\mathcal{M}^{p_{2},q_{2}}_{s}}
\lesssim
\|F\|_{L^{r_{1}}_{t}\mathcal{M}^{p_{1},q_{1}}_{s}}.
\end{align*}
\end{itemize}

Furthermore, we apply these Strichartz-type estimates to the study of nonlinear evolution equations associated with $\mathcal{A}_{k,l}^{\beta}$.

\begin{itemize}
\item In particular, we establish global well-posedness for the Hartree-type nonlinear heat equation
\begin{align*}
\begin{cases}
\partial_{t}u(t,x)+\mathcal{A}_{k,l}^{\beta}u(t,x)
=
\bigl(|x|^{-\gamma}*|u(t,x)|^{2}\bigr)u(t,x),\\[1mm]
u(0,x)=u_{0}(x),
\end{cases}
\end{align*}
where $\gamma>0$ and $(t,x)\in [0,\infty)\times\mathbb{R}^{d}$.

\item Using the algebra property of modulation spaces
$\mathcal{M}^{p,q}_{s}$ for
$s>\frac{d}{q'}$,
we derive suitable trilinear estimates that yield the existence and uniqueness of global solutions.

\item Moreover, by establishing refined trilinear estimates that do not rely on the algebra property of modulation spaces, we extend the global well-posedness theory to a significantly broader class of modulation spaces with regularity index $s\geq 0$.
\end{itemize}

\section*{Acknowledgements} The first author acknowledges support from the ANRF-ARG MATRICS Grant(002342).\\ The second author acknowledges the financial assistance provided by the University Grants Commission (UGC), India (File No. 231610192540) during the course of the Ph.D. programme.
\section*{Conflict of Interest}
The authors declare that there are no potential competing interests.

\bibliographystyle{abbrv}
\bibliography{ref}

\end{document}